\documentclass{article}%
\usepackage{graphicx}
\usepackage{amsmath}
\usepackage{amsfonts}
\usepackage{amssymb}
\usepackage{setspace}
\usepackage[hmargin=3.5cm,vmargin=3.5cm]{geometry}%
\providecommand{\U}[1]{\protect\rule{.1in}{.1in}}
\providecommand{\U}[1]{\protect\rule{.1in}{.1in}}
\providecommand{\U}[1]{\protect\rule{.1in}{.1in}}
\providecommand{\U}[1]{\protect\rule{.1in}{.1in}}
\providecommand{\U}[1]{\protect\rule{.1in}{.1in}}
\providecommand{\U}[1]{\protect\rule{.1in}{.1in}}
\providecommand{\U}[1]{\protect\rule{.1in}{.1in}}
\providecommand{\U}[1]{\protect\rule{.1in}{.1in}}
\providecommand{\U}[1]{\protect\rule{.1in}{.1in}}
\providecommand{\U}[1]{\protect\rule{.1in}{.1in}}
\providecommand{\U}[1]{\protect\rule{.1in}{.1in}}
\providecommand{\U}[1]{\protect\rule{.1in}{.1in}}
\providecommand{\U}[1]{\protect\rule{.1in}{.1in}}
\providecommand{\U}[1]{\protect\rule{.1in}{.1in}}
\providecommand{\U}[1]{\protect\rule{.1in}{.1in}}
\providecommand{\U}[1]{\protect\rule{.1in}{.1in}}
\providecommand{\U}[1]{\protect\rule{.1in}{.1in}}
\providecommand{\U}[1]{\protect\rule{.1in}{.1in}}
\providecommand{\U}[1]{\protect\rule{.1in}{.1in}}
\providecommand{\U}[1]{\protect\rule{.1in}{.1in}}
\providecommand{\U}[1]{\protect\rule{.1in}{.1in}}
\providecommand{\U}[1]{\protect\rule{.1in}{.1in}}
\providecommand{\U}[1]{\protect\rule{.1in}{.1in}}
\providecommand{\U}[1]{\protect\rule{.1in}{.1in}}
\providecommand{\U}[1]{\protect\rule{.1in}{.1in}}
\providecommand{\U}[1]{\protect\rule{.1in}{.1in}}
\providecommand{\U}[1]{\protect\rule{.1in}{.1in}}
\providecommand{\U}[1]{\protect\rule{.1in}{.1in}}
\providecommand{\U}[1]{\protect\rule{.1in}{.1in}}
\providecommand{\U}[1]{\protect\rule{.1in}{.1in}}
\providecommand{\U}[1]{\protect\rule{.1in}{.1in}}
\providecommand{\U}[1]{\protect\rule{.1in}{.1in}}
\providecommand{\U}[1]{\protect\rule{.1in}{.1in}}
\providecommand{\U}[1]{\protect\rule{.1in}{.1in}}
\providecommand{\U}[1]{\protect\rule{.1in}{.1in}}
\providecommand{\U}[1]{\protect\rule{.1in}{.1in}}
\providecommand{\U}[1]{\protect\rule{.1in}{.1in}}
\providecommand{\U}[1]{\protect\rule{.1in}{.1in}}

\newtheorem{theorem}{Theorem}
{}

\newtheorem{corollary}{Corollary}

\newtheorem{definition}{Definition}

\newtheorem{lemma}{Lemma}
{}

\newtheorem{proposition}{Proposition}
\newtheorem{remark}{Remark}

\newtheorem{summary}{Summary}
\newenvironment{proof}[1][Proof]{\textbf{#1.} }{\ \rule{0.5em}{0.5em}}

\begin{document}

\title{On the Spectrality and Spectral Expansion of the Non-self-adjoint
\ Mathieu-Hill Operator in $L_{2}(-\infty,\infty).$ }
\author{O. A. Veliev\\{\small Dogus University, Ac\i badem, Kadik\"{o}y, \ }\\{\small Istanbul, Turkey.}\ {\small e-mail: oveliev@dogus.edu.tr}}
\date{}
\maketitle

\begin{abstract}
In this paper we investigate the non-self-adjoint operator$\ H$ generated in
$L_{2}(-\infty,\infty)$ by the Mathieu-Hill equation with a complex-valued
potential. We find a necessary and sufficient conditions on the potential for
which $H$ has no spectral singularity at infinity and it is an asymptotically
spectral operator. Moreover, we give a detailed classification, stated in term
of the potential, for the form of the spectral decomposition of the operator
$H$ by investigating the essential spectral singularities.

Key Words: Mathieu-Hill operator, Spectral operator, Spectral Expansion.

AMS Mathematics Subject Classification: 34L05, 34L20.

\end{abstract}

\section{Introduction}

Let $L(q)$ be the Hill operator generated in $L_{2}(-\infty,\infty)$ by the expression%

\begin{equation}
l(y)=-y^{\prime\prime}+qy,
\end{equation}
where $q$ is a complex-valued summable function on $[0,1]$ and $q(x+1)=q(x)$
a.e.. It is well-known that (see [3, 8, 9]) the spectrum $S(L(q))$ of the
operator $L(q)$ is the union of the spectra $S(L_{t}(q))$ of the operators
$L_{t}(q)$ for $t\in(-\pi,\pi],$ where $L_{t}(q)$ is the operator generated in
$L_{2}[0,1]$ by (1) and the boundary conditions
\begin{equation}
y(1)=e^{it}y(0),\text{ }y^{\prime}(1)=e^{it}y^{\prime}(0).
\end{equation}
The spectrum of $L_{t}(q)$ for $t\in\mathbb{C}$ consist of the eigenvalues
that are the roots of
\begin{equation}
F(\lambda)=2\cos t,
\end{equation}
where $F(\lambda)=\varphi^{\prime}(1,\lambda)+\theta(1,\lambda)$, $\varphi$
and $\theta$ are the solutions of the equation $l(y)=\lambda y$ satisfying the
initial conditions $\theta(0,\lambda)=\varphi^{\prime}(0,\lambda)=1$
and$\quad\theta^{\prime}(0,\lambda)=\varphi(0,\lambda)=0.$

The operators $L_{t}(q)$ and $L(q)$ are denoted by $H_{t}(a,b)$ and $H(a,b)$
respectively when
\begin{equation}
\text{ }q(x)=ae^{-i2\pi x}+be^{i2\pi x},
\end{equation}
where $a$ and $b$ are the complex numbers. In this paper we consider the
spectrality and spectral expansion of the non-self-adjoint \ Mathieu-Hill
operator$\ H(a,b)$ defined in $L_{2}(-\infty,\infty)$. For this aim, first, in
Section 3 we obtain the uniform with respect to $t$ in some neighborhood of
$0$ and $\pi$ asymptotic formulas for the eigenvalues of the operators
$L_{t}(q)$ and $H_{t}(a,b)$. These formulas are the preliminary investigations
and have an auxiliary nature. Then, in Section 4 using these asymptotic
formulas, we find a necessary and sufficient condition, stated in term of
potential (4), for the asymptotic spectrality of the operator $H(a,b)$.
Finally, in Section 5 we classify in detail the form of the spectral expansion
of $H(a,b)$ in term of $a$ and $b.$

Gesztezy and Tkachenko [6] proved two versions of a criterion for the Hill
operator $L(q)$ with $q\in L_{2}[0,1]$ to be a spectral operator of scalar
type, in sense of Danford, one analytic and one geometric. The analytic
version was stated in term of the solutions of Hill's equation. The geometric
version of the criterion uses algebraic and geometric \ properties of the
spectra of periodic/antiperiodic and Dirichlet boundary value problems.

The problem of describing explicitly, for which potentials $q$ the Hill
operators $L(q)$ are spectral operators appeared to have been open for about
60 years. In paper [13] we found the explicit conditions on the potential $q$
such that $L(q)$ is an asymptotically spectral operator. In this paper we find
a criterion for asymptotic spectrality of $H(a,b)$ stated in term of $a$ and
$b$. Note that these investigations show that the set of potentials $q$ for
which $L(q)$ is spectral is a small subset of the periodic functions and it is
very hard to describe explicitly the required subset. Moreover, the papers
[17, 18] and this paper show that the investigation of the spectrality is
ineffective for the construction of the spectral expansion for $L(q).$ For
this in [17, 18] we introduced a new notions essential spectral singularity
(ESS) and ESS at infinity and proved that they determine the form of the
spectral expansion for $L(q).$ In this paper investigating the ESS and ESS at
infinity for $H(a,b)$ we classify the form of its spectral expansion in term
of $a$ and $b.$

To describe more precisely the main results of this paper let us introduce
some notations and definitions of the needed notions. The spectrum of
$L_{t}(q)$ consist of the eigenvalues. In [16] we proved that the eigenvalues
$\lambda_{n}(t)$ of $L_{t}$ can be numbered (counting the multiplicity) by
elements of $\mathbb{Z}$ such that, for each $n$ the function $\lambda
_{n}(\cdot)$ is continuous on $(-\pi,\pi]$ and $\left\vert \lambda_{\pm
n}(t)\right\vert \rightarrow\infty$ as $n\rightarrow\infty.$ The spectrum of
$L(q)$ is the union of the continuous curves $\Gamma_{n}=\{\lambda_{n}%
(t):t\in(-\pi,\pi]\}$ for $n\in\mathbb{Z}.$ Let $\Psi_{n,t}$ be the normalized
eigenfunction corresponding to the simple eigenvalue $\lambda_{n}(t)$\ and
$\Psi_{n,t}^{\ast}$ be the normalized eigenfunction of $(L_{t}(q))^{\ast}$
corresponding to $\overline{\lambda_{n}(t)}.$ It is well-known that (see p. 39
of [10]) if $\lambda_{n}(t)$ is a simple eigenvalue of $L_{t},$ then the
spectral projection $e(\lambda_{n}(t))$ defined by contour integration of the
resolvent of $L_{t}(q)$ over the closed contour containing only the eigenvalue
$\lambda_{n}(t),$ has the form%
\begin{equation}
e(\lambda_{n}(t))f=\frac{1}{d_{n}(t)}(f,\Psi_{n,t}^{\ast})\Psi_{n,t},\text{ }%
\end{equation}
where
\begin{equation}
d_{n}(t)=\left(  \Psi_{n,t},\Psi_{n,t}^{\ast}\right)  ,\text{ }\left\Vert
e\left(  \lambda_{n}(t)\right)  \right\Vert =\left\vert d_{n}(t)\right\vert
^{-1},
\end{equation}
and $\left(  \cdot,\cdot\right)  $ is the inner product in $L_{2}[0,1].$ Note
that in this paper the number $d_{n}(t)$ is defined only for the simple
eigenvalues $\lambda_{n}(t)$. If $\lambda_{n}(t)$ is a simple eigenvalue then
the normalized eigenfunctions $\Psi_{n,t}$ and $\Psi_{n,t}^{\ast}$ are
determined uniquely up to constant of modulus $1.$ Therefore $\mid
d_{n}(t)\mid$ is uniquely defined and it is the norm of the projection
$e\left(  \lambda_{n}(t)\right)  $. Note also that $\lambda_{n}(t)$ is a
simple eigenvalue if $F^{\prime}(\lambda_{n}(t))\neq0$ and the roots of the
equation $F^{\prime}(\lambda)=0$ is a discrete set, since $F^{\prime}$ is an
entire function. Thus $\lambda_{n}(t)$ is a simple eigenvalue for $t\in\left(
(-\pi,\pi]\backslash A_{n}\right)  $, where $A_{n}$ is at most a finite set.
Moreover $\mid d_{n}\mid$ is continuous at $t$ if $\lambda_{n}(t)$ is a simple
eigenvalue. Therefore in this paper we prefer the following definitions stated
in term of $d_{n}(t).$ McGarvey [8] proved that $L(q)$ is a spectral operator
if and only if there exists $c_{1}>0$ such that $\left\Vert e\left(
\lambda_{n}(t)\right)  \right\Vert <c_{1}$ for $n\in\mathbb{Z}$ and for almost
all $t\in(-\pi,\pi].$ It can be stated in terms of $d_{n}(t)$ and $A_{n}$ as follows.

\begin{definition}
We say that $L(q)$ is a spectral operator if there exists $c_{1}>0$ such that
\begin{equation}
\left\vert d_{n}(t)\right\vert ^{-1}<c_{1}%
\end{equation}
for all $n\in\mathbb{Z}$ and $t\in\left(  (-\pi,\pi]\backslash A_{n}\right)  $.
\end{definition}

Note that here and in subsequent relations we denote by $c_{i}$ for
$i=1,2,...$ the positive constants whose exact values are inessential.
Similarly, we use the following definition.

\begin{definition}
We say that $L(q)$ is an asymptotically spectral operator if there exists
$N>0$ such that (7) holds for all $\left\vert n\right\vert >N$ and
$t\in\left(  (-\pi,\pi]\backslash A_{n}\right)  $.
\end{definition}

As was noted in the paper [16], the spectral singularity of the operator
$L(q)$ are the points $\lambda\in S(L(q))$ for which the projections $e\left(
\lambda_{n}(t)\right)  $ corresponding to the simple eigenvalues $\lambda
_{n}(t)$ lying in some neighborhood of $\lambda$ are not uniformly bounded.
Therefore we have the following definitions for the spectral singularities in
term of $d_{n}.$

\begin{definition}
A point $\lambda\in\sigma(L(q))$ is said to be a spectral singularity of
$L(q)$ if there exist $n\in\mathbb{Z}$ and sequence $\{t_{k}\}\subset\left(
(-\pi,\pi]\backslash A_{n}\right)  $ such that $\lambda_{n}(t_{k}%
)\rightarrow\lambda$ and $\left\vert d_{n}(t_{k})\right\vert \rightarrow0$ as
$k\rightarrow\infty.$ We say that the operator $L(q)$ has a spectral
singularity at infinity if there exist sequences $\left\{  n_{k}\right\}
\subset\mathbb{Z}$ and $\left\{  t_{k}\right\}  \subset\left(  (-\pi
,\pi]\backslash A_{n_{k}}\right)  $ such that $\left\vert n_{k}\right\vert
\rightarrow\infty$ and $\left\vert d_{n_{k}}(t_{k})\right\vert \rightarrow0$
as $k\rightarrow\infty.$
\end{definition}

It is clear that the operator $L(q)$ has no the spectral singularity at
infinity if and only if it is asymptotically spectral operator. Now let us
list the main results.

\begin{theorem}
(Main Result for Spectrality) The operator $H(a,b)$ has no spectral
singularity at infinity and is an asymptotically spectral operator if and only
if $\mid a\mid=\mid b\mid$\ and
\begin{equation}
\text{ }\inf_{q,p\in\mathbb{N}}\{\mid q\alpha-(2p-1)\mid\}\neq0,
\end{equation}
where $\alpha=\pi^{-1}\arg(ab)$, $\mathbb{N=}\left\{  1,2,...,\right\}  .$
\end{theorem}

This main result of Section 4 implies the following

\begin{corollary}
Let $ab\in\mathbb{R}$. Then $H(a,b)$ is a spectral operator if and only it is
self adjoint.
\end{corollary}

These results show that the theory of spectral operator is ineffective for the
study of the spectral expansion for the non-self-adjoint operator $H(a,b)$
too. It was proven in [5] that in the self-adjoint case the spectral expansion
of $L(q)$ has the following elegant form
\begin{equation}
f=\frac{1}{2\pi}\sum_{n\in\mathbb{Z}}\int\nolimits_{(-\pi,\pi]}a_{n}%
(t)\Psi_{n,t}dt,
\end{equation}
where
\begin{equation}
a_{n}(t)=\frac{1}{\overline{d_{n}(t)}}\left(  \int\nolimits_{\mathbb{R}%
}f(x)\overline{\Psi_{n,t}^{\ast}(x)}dx\right)  .
\end{equation}
In the non-self-adjoint case to obtain the spectral expansion, we need to
consider the integrability of $a_{n}(t)\Psi_{n,t}$ with respect to $t$ over
$(-\pi,\pi]$ which is connected with the integrability of $\frac{1}{d_{n}}.$
Therefore in [17] we introduced the following notions.

\begin{definition}
A number $\lambda_{0}\in\sigma(L)$ is said to be an essential spectral
singularity (ESS) of $L$ if there exist $t_{0}\in(-\pi,\pi]$ and
$n\in\mathbb{Z}$ such that $\lambda_{0}=\lambda_{n}(t_{0})$ and $\frac
{1}{d_{n}}$ is not integrable over $(t_{0}-\delta,t_{0}+\delta)$ for all
$\delta>0.$
\end{definition}

It is clear that $\lambda_{0}=\lambda_{n}(t_{0})$ is ESS if and only is there
exists sequence of closed intervals $I(s)$ approaching $t_{0\text{ }}$such
that $\lambda_{n}(t)$ for $t\in I(s)$ are the simple eigenvalue and
\begin{equation}
\lim_{s\rightarrow\infty}\int\nolimits_{I(s)}\left\vert d_{n}(t)\right\vert
^{-1}dt=\infty.
\end{equation}

It the similar way in [17] we defined ESS at infinity.

\begin{definition}
We say that the operator $L(q)$ has ESS at infinity if there exist sequence of
integers $n_{s}$ and sequence of closed subsets $I(s)$ of $(-\pi,\pi]$ such
that $\lambda_{n_{s}}(t)$ for $t\in I(s)$ are the simple eigenvalues and
\begin{equation}
\lim_{s\rightarrow\infty}\int\nolimits_{I(s)}\left\vert d_{n_{s}%
}(t)\right\vert ^{-1}dt=\infty.
\end{equation}

\end{definition}

Note that it follows from the above definitions that the boundlessness of
$\left\vert d_{n}(\cdot)\right\vert ^{-1}$ is the characterization of the
spectral singularities and the considerations of the spectral singularities
play only the crucial rule for the investigations of the spectrality of
$L(q)$. On the other hand, the periodic differential operators, in general, is
not a spectral operator. Therefore to construct the spectral expansion for the
operator $L,$ in the general case, in [17, 18] we introduced the new concepts
ESS which connected with the nonintegrability of $\left\vert d_{n}%
(\cdot)\right\vert ^{-1}$ and proved that the spectral expansion has the
elegant form (9) if and only if $L(q)$ has no ESS and ESS at infinity. In
Section 5 investigating the ESS and ESS at infinity for $H(a,b)$ we obtained
the following main results for its spectral expansion.

\begin{theorem}
If $0<\left\vert ab\right\vert <16/9$, then $H(a,b)$ has no ESS and ESS at
infinity and its spectral expansion has the elegant form (9).
\end{theorem}

For the largest subclass of the potentials (4) we prove the following criterion

\begin{theorem}
The non-self-adjoint operator $H(a,b)$ has no ESS at infinity, has at most
finite number of ESS and its spectral expansion has the asymptotically elegant
form
\begin{equation}
f(x)=\frac{1}{2\pi}\left(  \int\nolimits_{(-\pi,\pi]}\sum\limits_{n\in
\mathbb{S}}a_{n}(t)\Psi_{n,t}(x)dt+\sum_{n\in\mathbb{Z}\backslash\mathbb{S}%
}\int\nolimits_{(-\pi,\pi]}a_{n}(t)\Psi_{n,t}(x)dt\right)
\end{equation}
if and only if $ab\neq0,$ where $\mathbb{S}$ is at most a finite set and is
the set of the indices $n$ for which $\Gamma_{n}$ contains at least one ESS.
\end{theorem}

For the remaining potentials we prove the following criterion

\begin{theorem}
The operator $H(a,b)$ has ESS at infinity and infinitely many ESS and its
spectral expansion has the following form if and only if either $a=0$ or
$b=0$
\begin{equation}
f(x)=f_{0}(x)+f_{\pi}(x)+\frac{1}{2\pi}%
{\displaystyle\sum\limits_{n\in\mathbb{Z}}}
\int\nolimits_{B(h)}a_{n}(t)\Psi_{n,t}(x)dt,
\end{equation}
where%
\begin{equation}
2\pi f_{0}(x)=\int\nolimits_{[-h,h]}a_{0}(t)\Psi_{0,t}(x)dt+\sum_{n=1}%
^{\infty}\int\nolimits_{[-h,h]}\left(  a_{-n}(t)\Psi_{-n,t}(x)+a_{n}%
(t)\Psi_{n,t}(x)\right)  dt,
\end{equation}%
\begin{equation}
2\pi f_{\pi}(x)=\sum_{n=0}^{\infty}\int\nolimits_{[\pi-h,\pi+h]}\left(
a_{n}(t)\Psi_{n,t}(x)+a_{-(n+1)}(t)\Psi_{-(n+1),t}(x)\right)  dt,
\end{equation}
$B(h)=[h,\pi-h]\cup\lbrack\pi+h,2\pi-h],$ $0<h<\frac{1}{15\pi}$.
\end{theorem}

Note that if the conditions requested for $H(a,b)$ in Theorem 3 do not hold
then either $a=0$ or $b=0,$ that is, the conditions requested for $H(a,b)$ in
Theorem 4 hold. It means that all cases of the potential (4) are investigated
in Theorem 3 and Theorem 4. In Theorem 2 some subcase of Theorem 3 is studied.

\section{Preliminary Facts}

In this section we present some results of [12, 13, 2] which are used in this paper.

\textbf{Theorem 2 of [12].}\textit{ The eigenvalues }$\lambda_{n}(t)$\textit{
\ and eigenfunctions }$\Psi_{n,t}$\textit{ of the operators }$L_{t}%
(q)$\textit{ for }$t\neq0,\pi,$\textit{ satisfy the following asymptotic
formulas }%
\begin{equation}
\lambda_{n}(t)=(2\pi n+t)^{2}+O(n^{-1}\ln\left\vert n\right\vert ),\text{
}\Psi_{n,t}(x)=e^{i(2\pi n+t)x}+O(n^{-1}).
\end{equation}
\textit{for }$\left\vert n\right\vert \rightarrow\infty.$\textit{ For any
fixed number }$\rho\in(0,\pi/2),$\textit{ these asymptotic formulas are
uniform with respect to }$t$\textit{ in }$[\rho,\pi-\rho]$\textit{. Moreover,
there exists a positive number }$N(\rho),$\textit{ independent of }%
$t,$\textit{ such that the eigenvalues }$\lambda_{n}(t)$\textit{ for \ }%
$t\in\lbrack\rho,\pi-\rho]$\textit{ and }$\mid n\mid>N(\rho)$\textit{ are
simple.}

Note that, the formula\ $f(n,t)=O(h(n))$ is said to be uniform with respect to
$t$ in a set $I$ if there exist positive constants $M$ and $N,$ independent of
$t,$ such that $\mid f(n,t))\mid<M\mid h(n)\mid$ for all $t\in I$ and $\mid
n\mid\geq N.$ We use Remark 2.1 and lot of formulas of [13] that are listed in
Remark 1 and as formulas (20)-(36).

\begin{remark}
In Remark 2.1 of [13] \ we proved that here exists a positive integer $N(0)$
such that the disk $U(n,t,\rho)=:\{\lambda\in\mathbb{C}:\left\vert
\lambda-(2\pi n+t)^{2}\right\vert \leq15\pi n\rho\}$ for $t\in\lbrack0,\rho],$
where $15\pi\rho<1,$ and $n>N(0)$ contains two eigenvalues (counting with
multiplicities) denoted by $\lambda_{n,1}(t)$ and $\lambda_{n,2}(t)$ and these
eigenvalues can be chosen as a continuous function of $t$ on the interval
$[0,\rho].$ Similarly, there exists a positive integer $N(\pi)$ such that the
disk $U(n,t,\rho)$ for $t\in\lbrack\pi-\rho,\pi]$ and $n>N(\pi)$ contains two
eigenvalues (counting with multiplicities) denoted again by $\lambda_{n,1}(t)$
and $\lambda_{n,2}(t)$ that are continuous function of $t$ on the interval
$[\pi-\rho,\pi].$

Thus for $n>$ $N=:\max\left\{  N(\rho),N(0),N(\pi)\right\}  ,$ the eigenvalues
$\lambda_{n,1}(t)$ and $\lambda_{n,2}(t)$ are continuous on $[0,\rho
]\cup\lbrack\pi-\rho,\pi]$ and for$\mid n\mid>N$ \ \textit{the eigenvalue
}$\lambda_{n}(t),$\textit{ defined by (17), is continuous on \ }$[\rho
,\pi-\rho].$\textit{ Moreover, }$\lambda_{n,1}$ and $\lambda_{n,2}$ can be
chosen so that\textit{ }%
\begin{equation}
\lambda_{n,1}(\rho)=\lambda_{-n}(\rho),\text{ }\lambda_{n,2}(\rho)=\lambda
_{n}(\rho),
\end{equation}%
\begin{equation}
\lambda_{n,1}(\pi-\rho)=\lambda_{n}(\pi-\rho),\text{ }\lambda_{n,2}(\pi
-\rho)=\lambda_{-(n+1)}(\pi-\rho).
\end{equation}

Let us redenote $\lambda_{n,1}(t)$ and $\,\lambda_{n,2}(t)$ by $\lambda
_{-n}(t)$ and $\lambda_{n}(t)$ respectively for $n>N$ and $t\in\lbrack
0,\rho].$ Similarly, redenote $\lambda_{n,1}(t)$ and $\,\lambda_{n,2}(t)$ by
$\lambda_{n}(t)$ and $\lambda_{-n-1}(t)$ respectively for $n>N$ and
$t\in\lbrack\pi-\rho,\pi].$ Defining $\lambda_{n}(-t)=\lambda_{n}(t)$ we
obtain continuous function $\lambda_{n}$ on $(-\pi,\pi].$ In this paper we use
both notations: $\lambda_{n}(t)$ and $\lambda_{n,j}(t)$.
\end{remark}

One can readily see that
\begin{equation}
\left\vert \lambda-(2\pi(n-k)+t)^{2}\right\vert >\left\vert k\right\vert
\left\vert 2n-k\right\vert ,\text{ \ }\forall\lambda\in U(n,t,\rho)
\end{equation}
for $k\neq0,2n$ and $t\in\lbrack0,\rho]$, where $\left\vert n\right\vert >N$.

In [13] to obtain the uniform, with respect to $t\in\lbrack0,\rho],$
asymptotic formulas for the eigenvalues $\lambda_{n,j}(t)$ we used (20) and
the iteration of the formula
\begin{equation}
(\lambda_{n,j}(t)-(2\pi n+t)^{2})(\Psi_{n,j,t},e^{i(2\pi n+t)x})=(q\Psi
_{n,j,t},e^{i(2\pi n+t)x}),
\end{equation}
where $\Psi_{n,j,t}$ is any normalized eigenfunction corresponding to
$\lambda_{n,j}(t).$ Iterating (21) infinite times we got the following
formula
\begin{equation}
(\lambda_{n,j}(t)-(2\pi n+t)^{2}-A(\lambda_{n,j}(t),t))u_{n,j}(t)=(q_{2n}%
+B(\lambda_{n,j}(t),t))v_{n,j}(t),
\end{equation}
where $u_{n,j}(t)=(\Psi_{n,j,t},e^{i(2\pi n+t)x}),$ $v_{n,j}(t)=(\Psi
_{n,j,t},e^{i(-2\pi n+t)x}),$ $q_{n}=(q,e^{i2\pi nx}),$
\begin{equation}
A(\lambda,t)=\sum_{k=1}^{\infty}a_{k}(\lambda,t),\text{ }B(\lambda
,t)=\sum_{k=1}^{\infty}b_{k}(\lambda,t),
\end{equation}%
\begin{equation}
a_{k}(\lambda,t)=\sum_{n_{1},n_{2},...,n_{k}}q_{-n_{1}-n_{2}-...-n_{k}}%
{\textstyle\prod\limits_{s=1}^{k}}
q_{n_{s}}\left(  \lambda-(2\pi(n-n_{1}-..-n_{s})+t)^{2}\right)  ^{-1},
\end{equation}%
\begin{equation}
b_{k}(\lambda,t)=\sum_{n_{1},n_{2},...,n_{k}}q_{2n-n_{1}-n_{2}-...-n_{k}}%
{\textstyle\prod\limits_{s=1}^{k}}
q_{n_{s}}\left(  \lambda-(2\pi(n-n_{1}-..-n_{s})+t)^{2}\right)  ^{-1}%
\end{equation}
for $\lambda\in U(n,t,\rho)$ (see (37) of [13]).

Similarly, we obtained the formula
\begin{equation}
(\lambda_{n,j}(t)-(-2\pi n+t)^{2}-A^{\prime}(\lambda_{n,j}(t),t))v_{n,j}%
(t)=(q_{-2n}+B^{\prime}(\lambda_{n,j}(t),t))u_{n,j}(t),
\end{equation}
where%
\begin{equation}
A^{\prime}(\lambda,t)=\sum_{k=1}^{\infty}a_{k}^{\prime}(\lambda,t),\text{
}B^{\prime}(\lambda,t)=\sum_{k=1}^{\infty}b_{k}^{\prime}(\lambda,t),
\end{equation}%
\begin{equation}
a_{k}^{\prime}(\lambda,t)=\sum_{n_{1},n_{2},...,n_{k}}q_{-n_{1}-n_{2}%
-...-n_{k}}%
{\textstyle\prod\limits_{s=1}^{k}}
q_{n_{s}}\left(  \lambda-(2\pi(n+n_{1}+..+n_{s})-t)^{2}\right)  ^{-1},
\end{equation}%
\begin{equation}
b_{k}^{\prime}(\lambda,t)=\sum_{n_{1},n_{2},...,n_{k}}q_{-2n-n_{1}%
-n_{2}-...-n_{k}}%
{\textstyle\prod\limits_{s=1}^{k}}
q_{n_{s}}\left(  \lambda-(2\pi(n+n_{1}+..+n_{s})-t)^{2}\right)  ^{-1}%
\end{equation}
for $\lambda\in U(n,t,\rho)$ (see (38) of [13]).

The sums in (24), (25) and (28), (29) are taken under conditions $n_{1}%
+n_{2}+...+n_{s}\neq0,2n$ and $n_{1}+n_{2}+...+n_{s}\neq0,-2n$ respectively,
where $s=1,2,...$

Besides, it was proved [13] that the equalities
\begin{equation}
a_{k}(\lambda,t),\text{ }b_{k}(\lambda,t),\text{ }a_{k}^{\prime}%
(\lambda,t),\text{ }b_{k}^{\prime}(\lambda,t)=O\left(  (n^{-1}\ln\left\vert
n\right\vert )^{k}\right)
\end{equation}
hold uniformly for $t\in\lbrack0,\rho]$ and $\lambda\in U(n,t,\rho)$\ (see
(34) and (36) of [13]), and derivatives of these functions with respect to
$\lambda$ are $O(n^{-k-1})$ (see the proof of Lemma 2.5) which imply that the
functions $A(\lambda,t),$ $A^{\prime}(\lambda,t),$ $B(\lambda,t)$ and
$B^{\prime}(\lambda,t)$ are analytic on $U(n,t,\rho).$ Moreover, there exists
a constant $K$ such that
\begin{equation}
\mid A(\lambda,t)\mid<Kn^{-1},\text{ }\mid A^{\prime}(\lambda,t)\mid
<Kn^{-1},\text{ }\mid B(\lambda,t)\mid<Kn^{-1},\text{ }\mid B^{\prime}%
(\lambda,t)\mid<Kn^{-1},
\end{equation}%
\begin{equation}
\mid A(\lambda,t)-A(\mu,t)\mid<Kn^{-2}\mid\lambda-\mu\mid,\text{ }\mid
A^{\prime}(\lambda,t)-A^{\prime}(\mu,t)\mid<Kn^{-2}\mid\lambda-\mu\mid,
\end{equation}%
\begin{equation}
\mid B(\lambda,t)-B(\mu,t)\mid<Kn^{-2}\mid\lambda-\mu\mid,\text{ }\mid
B^{\prime}(\lambda,t)-B^{\prime}(\mu,t)\mid<Kn^{-2}\mid\lambda-\mu\mid,
\end{equation}%
\begin{equation}
\mid C(\lambda,t)\mid<tKn^{-1},\text{ }\mid C(\lambda,t))-C(\mu,t))\mid
<tKn^{-2}\mid\lambda-\mu\mid
\end{equation}
for all $\ n>N,$ $t\in\lbrack0,\rho]$ and $\lambda,\mu\in U(n,t,\rho),$ where
$N$ and $U(n,t,\rho)$ are defined in Remark 1, and
\[
C(\lambda,t)=\frac{1}{2}(A(\lambda,t)-A^{\prime}(\lambda,t))
\]
\textbf{(}see Lemma 2.3 and Lemma 2.5 of [13]\textbf{).}

In this paper we use also the following, uniform with respect to $t\in
\lbrack0,\rho],$ equalities from [13] (see (26)-(28) of [13]) for the
normalized eigenfunction $\Psi_{n,j,t}$:
\begin{equation}
\Psi_{n,j,t}(x)=u_{n,j}(t)e^{i(2\pi n+t)x}+v_{n,j}(t)e^{i(-2\pi n+t)x}%
+h_{n,j,t}(x),
\end{equation}%
\begin{equation}
(h_{n,j,t},e^{i(\pm2\pi n+t)x})=0,\text{ }\left\Vert h_{n,j,t}\right\Vert
=O(n^{-1}),\text{ }\left\vert u_{n,j}(t)\right\vert ^{2}+\left\vert
v_{n,j}(t)\right\vert ^{2}=1+O(n^{-2}).
\end{equation}
Here we also use formula (55) of [2] about estimations of $B(\lambda,0)$ and
$B^{\prime}(\lambda,0)$ as follows: \ 

\textit{ Let the potential has the form (4), }$\lambda=(2\pi n)^{2}%
+z,$\textit{ where \ }$\left\vert z\right\vert <1,$\textit{ and }%
\begin{equation}
p_{n_{1},n_{2},...,n_{k}}(\lambda,0)=q_{2n-n_{1}-n_{2}-...-n_{k}}%
{\textstyle\prod\limits_{s=1}^{k}}
q_{n_{s}}\left(  \lambda-(2\pi(n-n_{1}-..-n_{s}))^{2}\right)  ^{-1}%
\end{equation}
\textit{be summands of }$b_{k}(\lambda,t)$\textit{ for }$t=0$\textit{ (see
(14)). Then using }(55) of [2] \textit{with }$q\geq2$\textit{ and the
estimation}
\[
\sum_{q\geq2}\left(  _{q}^{n+2q}\right)  \left(  \frac{\left\vert
ab\right\vert }{n^{2}}\right)  ^{q}=O(n^{-2})
\]
\textit{of [2] (see the estimation after formula (55) of [2]) and taking into
account that if }$k$\textit{ changes from }$2n+3$\textit{ to }$\infty
,$\textit{ then the number of steps}$=-2$\textit{ (that is, in our notations
the number of indices }$n_{1},n_{2},...,n_{k}$\textit{ of (37) that are equal
to }$1$\textit{) changes from }$2$\textit{ to }$\infty,$\textit{ we obtain}%

\begin{equation}
\sum_{k=2n+3}^{\infty}\sum_{n_{1},n_{2},...,n_{k}}\left\vert p_{n_{1}%
,n_{2},...,n_{k}}(\lambda,0)\right\vert =b_{2n-1}(\lambda,0)O(n^{-2}).
\end{equation}

\section{On the Operators $L_{t}(q)$ and $H_{t}(a,b).$}

One can readily see from \ (22), (26), (31) and Remark 1 that
\begin{equation}
\lambda_{n,j}(t)\in\left(  d^{-}(r(n),t)\cup d^{+}(r(n),t)\right)  \subset
U(n,t,\rho),
\end{equation}
for all $\ n>N,$ $t\in\lbrack0,\rho],$ where $r(n)=\max\{\left\vert
q_{2n}\right\vert ,\left\vert q_{-2n}\right\vert \}+2Kn^{-1}$ and $\ d^{\pm
}(r(n),t)$ is the disk with center $(\pm2\pi n+t)^{2}$ and radius $r(n).$
Indeed if $\left\vert u_{n,j}(t)\right\vert \geq\left\vert v_{n,j}%
(t)\right\vert $ then using (22), if $\left\vert v_{n,j}(t)\right\vert
>\left\vert u_{n,j}(t)\right\vert ,$ then using (26) and (31) we get (39).

\begin{theorem}
\textit{ A number }$\lambda\in U(n,t,\rho)$ is an eigenvalue of $L_{t}(q)$ for
$t\in\lbrack0,\rho]$ and $n>N,$ where $U(n,t,\rho)$ and $N$ are defined in
Remark \ 1, if and only if
\begin{equation}
(\lambda-(2\pi n+t)^{2}-A(\lambda,t))(\lambda-(2\pi n-t)^{2}-A^{\prime
}(\lambda,t))=(q_{2n}+B(\lambda,t))(q_{-2n}+B^{\prime}(\lambda,t)).
\end{equation}
Moreover\textit{ }$\lambda\in U(n,t,\rho)$ is a double eigenvalue of $L_{t}$
if and only if it is a double root of (40).
\end{theorem}

\begin{proof}
If $u_{n,j}(t)=0,$ then by (36) we have$\ v_{n,j}(t)\neq0$.\ Therefore, (22)
and (26) imply that
\[
q_{2n}+B(\lambda_{n,j}(t),t)=0,\text{ }\lambda_{n,j}(t)-(-2\pi n+t)^{2}%
-A^{\prime}(\lambda_{n,j}(t),t)=0,
\]
that is, the right-hand side and the left-hand side of (40)\ vanish \ when
$\lambda$ is replaced by$\ \lambda_{n,j}(t)$. Hence $\lambda_{n,j}(t)$
satisfies (40). In the same way we prove that if $\ \ v_{n,j}(t)=0$ then
$\lambda_{n,j}(t)$ is a \ root of (40). It remains to consider the case
$u_{n,j}(t)v_{n,j}(t)\neq0.$ In this case multiplying (22) and (26) side by
side and canceling $u_{n,j}(t)v_{n,j}(t)$ we get an equality obtained from
(40) by replacing $\lambda$ with $\lambda_{n,j}(t).$\ Thus, in any case
$\lambda_{n,j}(t)$ is a root of\ (40).

Now we prove that the roots of (40) lying in $U(n,t,\rho)$\ are the
eigenvalues of $L_{t}(q).$ Let $F(\lambda,t)$ be the left-hand side minus the
right-hand side of (40). Using (31) one can easily verify that the inequality
\
\begin{equation}
\mid F(\lambda,t)-G(\lambda,t)\mid<\mid G(\lambda,t)\mid,
\end{equation}
where $G(\lambda,t)=(\lambda-(2\pi n+t)^{2})(\lambda-(2\pi n-t)^{2}),$ holds
for all $\lambda$ from the boundary of $U(n,t,\rho).$ Since the function
$(\lambda-(2\pi n+t)^{2})(\lambda-(2\pi n-t)^{2})$ has two roots in \ the
\ set $U(n,t,\rho),$ by the Rouche's theorem from (41) we obtain that
$F(\lambda,t)$ has two roots in the same\ set.\ Thus\ $L_{t}(q)$ has two
eigenvalue (counting with multiplicities) lying in $U(n,t,\rho)$ (see Remark
1) that are the roots of (40). On the other hand, (40) has preciously two
roots (counting with multiplicities) in $U(n,t,\rho).$ Therefore $\lambda\in
U(n,t,\rho)$ is an eigenvalue of $L_{t}(q)$ if and only if (40) holds.

If \textit{ }$\lambda\in U(n,t,\rho)$ is a double eigenvalue of $L_{t}(q),$
then by Remark 1 $L_{t}(q)$ has no other eigenvalues\textit{ }in\textit{
}$U(n,t,\rho)$ and hence (40) has no other roots. This implies that $\lambda$
is a double root of (40). By the same argument one can prove that if $\lambda$
is a double root of (40) then it is a double eigenvalue of $L_{t}(q)$
\end{proof}

One can readily verify that equation (40) can be written in the form
\begin{equation}
(\lambda-(2\pi n+t)^{2}-\frac{1}{2}(A+A^{\prime})+4\pi nt)^{2}=D,
\end{equation}
where
\begin{equation}
D(\lambda,t)=(4\pi nt)^{2}+q_{2n}q_{-2n}+8\pi ntC+C^{2}+q_{2n}B^{\prime
}+q_{-2n}B+BB^{\prime}%
\end{equation}
and, for brevity, we denote $C(\lambda,t),$ $\ B(\lambda,t),$ $\ A(\lambda,t)$
etc. by $C,$ $\ B,$ $A$ etc. It is clear that $\lambda$ is a root of (42) if
and only if \ it satisfies at least one of the equations
\begin{equation}
\lambda-(2\pi n+t)^{2}-\frac{1}{2}(A(\lambda,t)+A^{\prime}(\lambda,t))+4\pi
nt=-\sqrt{D(\lambda,t)}%
\end{equation}
and
\begin{equation}
\lambda-(2\pi n+t)^{2}-\frac{1}{2}(A(\lambda,t)+A^{\prime}(\lambda,t))+4\pi
nt=\sqrt{D(\lambda,t)},
\end{equation}
where
\begin{equation}
\sqrt{D}=\sqrt{\left\vert D\right\vert }e^{(\arg D)/2},\text{ }-\pi<\arg
D\leq\pi.
\end{equation}

\begin{remark}
It is clear from the construction of $D(\lambda,t)$ that this function is
continuous with respect to $(\lambda,t)$ for $t\in\lbrack0,\rho]$ and
$\lambda\in U(n,t,\rho).$ Moreover, by Remark 1 the eigenvalues $\lambda
_{n,1}(t)$ and $\lambda_{n,2}(t)$ continuously depend on $t\in\lbrack0,\rho].$
Therefore $D(\lambda_{n,j}(t),t)$ for $n>N$ and $j=1,2$ is a continuous
functions of $t\in\lbrack0,\rho].$ By (43), (34), (23), (27) and (30) we have
\[
D(\lambda_{n,j}(\rho),\rho)=(4\pi nt)^{2}+o(1),\text{ }A(\lambda_{n,j}%
(\rho),\rho)+A^{\prime}(\lambda_{n,j}(\rho),\rho)=o(1)
\]
as $n\rightarrow\infty.$ Therefore by (18) and \textit{Theorem 2 of [12]} the
eigenvalues $\lambda_{n,1}(\rho)$ and $\lambda_{n,2}(\rho)$ are simple,
$\lambda_{n,1}(\rho),$ satisfies (44) and $\lambda_{n,2}(\rho)$ satisfies
(45). If $\lambda_{n,1}(t)$ and $\lambda_{n,2}(t)$ are simple for $t\in\lbrack
t_{0},\rho],$ where $0\leq t_{0}\leq\rho,$ then these functions are analytic
function on $[t_{0},\rho]$ and $\lambda_{n,1}(t)\neq\lambda_{n,2}(t)$ for all
$t\in\lbrack t_{0},\rho]$.
\end{remark}

\begin{theorem}
Suppose that $\sqrt{D(\lambda_{n,j}(t),t))}$ continuously depends on $t$ at
$[t_{0},\rho]$ and
\begin{equation}
D(\lambda_{n,j}(t),t)\neq0,\text{ }\forall t\in\lbrack t_{0},\rho]
\end{equation}
for $n>N$ and $j=1,2,$ where $\rho$ and $N$ are defined in Remark \ 1 and
$\sqrt{D}$ is defined in (46) and $0\leq t_{0}\leq\rho$. Then for $t\in\lbrack
t_{0},\rho]$ the eigenvalues $\lambda_{n,1}(t)$ and $\lambda_{n,2}(t)$ defined
in Remark 1 are simple, $\lambda_{n,1}(t)$ satisfies (44) and $\lambda
_{n,2}(t),$ satisfies (45). That is
\begin{equation}
\lambda_{n,j}(t)=(2\pi n+t)^{2}+\frac{1}{2}(A(\lambda_{n,j},t)+A^{\prime
}(\lambda_{n,j},t))-4\pi nt+(-1)^{j}\sqrt{D(\lambda_{n,j},t)}%
\end{equation}
for $t\in\lbrack t_{0},\rho],$ $n>N$ and $j=1,2.$
\end{theorem}

\begin{proof}
By Remark 2, the eigenvalues $\lambda_{n,1}(\rho)$ and $\lambda_{n,2}(\rho)$
are simple, $\lambda_{n,1}(\rho)$ satisfies (44) and $\lambda_{n,2}(\rho)$
satisfies (45). Let us we prove that $\lambda_{n,1}(t)$ satisfies (44) for all
$t\in\lbrack t_{0},\rho]$. Suppose to the contrary that this claim is not
true. Then there exists $t\in\lbrack t_{0},\rho)$ and the sequences
$p_{n}\rightarrow t$ and $q_{n}\rightarrow t,$ where one of them may be a
constant sequence, such that $\lambda_{n,1}(p_{n})$ and $\lambda_{n,1}(q_{n})$
satisfy (44) and (45) respectively. Using the continuity of $\sqrt
{(D(\lambda_{n,j}(t),t))}$, we conclude that $\lambda_{n,1}(t)$ satisfies both
(44) and (45). However, it is possible only if $D(\lambda_{n,1}(t),t)=0$ which
contradicts (47). Hence $\lambda_{n,1}(t)$ satisfies (44) for all $t\in\lbrack
t_{0},\rho]$. In the same way we prove that $\lambda_{n,2}(t)$ satisfies (45)
for all $t\in\lbrack t_{0},\rho]$. If $\lambda_{n,1}(t)=\lambda_{n,2}(t)$ for
some value of $t\in\lbrack t_{0},\rho]$, that is if $\lambda_{n,j}(t)$ is a
double eigenvalue then it satisfies both (44) and (45) which again contradicts (47)
\end{proof}

Now we study the operator $H_{t}.$ Note that we consider only the case
$t\in\lbrack0,\rho]$ due to the following reason. The case $t\in\lbrack
\rho,\pi-\rho]$ was considered in [12]. \ The case $t\in\lbrack\pi-\rho,\pi]$
is similar to the case $t\in\lbrack0,\rho]$ and we explain it in Remark 3.
Besides, the eigenvalues of $H_{-t}$ coincides with the eigenvalues of
$H_{t}.$

When the potential $q$ has the form (4) then
\begin{equation}
q_{-1}=a,\text{ }q_{1}=b,\text{ }q_{n}=0,\text{ }\forall n\neq\pm1
\end{equation}
and hence the formulas (22), (26), (42) and (43) have the form
\begin{equation}
(\lambda_{n,j}(t)-(2\pi n+t)^{2}-A(\lambda_{n,j}(t),t))u_{n,j}(t)=B(\lambda
_{n,j}(t),t)v_{n,j}(t),
\end{equation}%
\begin{equation}
(\lambda_{n,j}(t)-(-2\pi n+t)^{2}-A^{\prime}(\lambda_{n,j}(t),t))v_{n,j}%
(t)=B^{\prime}(\lambda_{n,j}(t),t)u_{n,j}(t),
\end{equation}%
\begin{equation}
(\lambda-(2\pi n+t)^{2}-\frac{1}{2}(A(\lambda,t)+A^{\prime}(\lambda,t))+4\pi
nt)^{2}=D(\lambda,t),
\end{equation}%
\begin{equation}
D(\lambda,t)=(4\pi nt+C(\lambda,t))^{2}+B(\lambda,t)B^{\prime}(\lambda,t).
\end{equation}
Moreover, by Theorem 5,\textit{ }$\lambda\in U(n,t,\rho)$ is a double
eigenvalue of $H_{t}$ if and only if it satisfies (52) and the equation
\begin{equation}
2(\lambda-(2\pi n+t)^{2}-\frac{1}{2}(A+A^{\prime})+4\pi nt)^{2}(1-\frac{1}%
{2}\frac{\partial}{\partial\lambda}(A+A^{\prime}))=\frac{\partial}%
{\partial\lambda}(D(\lambda,t)).
\end{equation}
By, (39) and (49) $\lambda_{n,j}(t)\in\left(  d^{-}(2Kn^{-1},t)\cup
d^{+}(2Kn^{-1},t)\right)  \subset U(n,t,\rho).$ Therefore the formula
\begin{equation}
\lambda_{n,j}(t)=(2\pi n)^{2}+O(n^{-1})
\end{equation}
holds uniformly, with respect to $t\in\lbrack0,n^{-2}],$ for $j=1,2,$ i.e.,
there exist positive constants $M$ and $N$ such that $\mid\lambda
_{n,j}(t)-(2\pi n)^{2}\mid<Mn^{-1}$for $n\geq N$ \ and $t\in\lbrack0,n^{-2}].$

Let us consider the functions taking part in (50)-(52). From (49) we see that
the indices in formulas (24), (25) for the case (4) satisfy the conditions
\begin{equation}
\{n_{1},n_{2},...,n_{k}\}\subset\{-1,1\},\text{ }n_{1}+n_{2}+...+n_{s}%
\neq0,2n,
\end{equation}%
\begin{equation}
\{n_{1},n_{2},...,n_{k},2n-n_{1}-n_{2}-...-n_{k}\}\subset\{-1,1\},\text{
}n_{1}+n_{2}+...+n_{s}\neq0,2n
\end{equation}
for $s=1,2,...,k$ respectively. Hence, by (49) $q_{-n_{1}-n_{2}-...-n_{k}}=0$
if $k$ is an even number. Therefore, by (24) and (28)%
\begin{equation}
a_{2m}(\lambda,t)=0,\text{ }a_{2m}^{\prime}(\lambda,t)=0,\text{ }\forall
m=1,2,...
\end{equation}
Since the indices $n_{1},n_{2},...,n_{k}$ take two values (see (56)) the
number of the summands in the right-hand side of (24) is not more than
$2^{k}.$ Clearly, these summands for $k=2m-1$ have the form
\[
a_{k}(\lambda,n_{1},n_{2},...,n_{k},t)=:(ab)^{m}%
{\textstyle\prod\limits_{s=1,2,...,k}}
\left(  \lambda-(2\pi(n-n_{1}-n_{2}-...-n_{s})+t)^{2}\right)  ^{-1}%
\]
(see (24), (49) and (56)). Therefore, we have
\begin{equation}
\text{ }a_{2m-1}(\lambda_{n,j}(t),t)=(4ab)^{m}O(n^{-2m+1}).
\end{equation}
If $t\in\lbrack0,n^{-2}],$ then one can readily see that
\[
a_{1}(\lambda_{n,j},t)=\frac{ab}{(2\pi n)^{2}+O(n^{-1})-(2\pi(n-1))^{2}}%
+\frac{ab}{(2\pi n)^{2}+O(n^{-1})-(2\pi(n+1))^{2}}%
\]%
\[
=\frac{ab}{2\pi(2\pi(2n-1)}-\frac{ab}{2\pi(2\pi(2n+1)}+O\left(  \frac{1}%
{n^{3}}\right)  =O\left(  \frac{1}{n^{2}}\right)  .
\]
The same estimations for $a_{2m-1}^{\prime}(\lambda_{n,j}(t),t)$ and
$a_{1}^{\prime}(\lambda_{n,j}(t),t)$ hold respectively. Thus, by (23), (27),
(30) and (58), we have
\begin{equation}
A(\lambda_{n,j}(t),t)=O(n^{-2}),\text{ }A^{\prime}(\lambda_{n,j}%
(t),t)=O(n^{-2}),\text{ }\forall t\in\lbrack0,n^{-2}].\text{ }%
\end{equation}

Now we study the functions $B(\lambda,t)$ and $B^{\prime}(\lambda,t)$ (see
(23), (25) and (27), (29)). First let us consider $b_{2n-1}(\lambda,t).$ If
$k=2n-1,$ then by (57) $n_{1}=n_{2}=...=n_{2k-1}=1.$ Using this and (49) in
(25) for $k=2n-1,$ we obtain
\begin{equation}
b_{2n-1}(\lambda,t)=b^{2n}%
{\textstyle\prod\limits_{s=1}^{2n-1}}
\left(  \lambda-(2\pi(n-s)+t)^{2}\right)  ^{-1}.
\end{equation}
If\ $k<2n-1$ or $k=2m,$ then, by (49), $q_{2n-n_{1}-n_{2}-...-n_{k}}=0$ and by
(25)
\begin{equation}
\text{ }b_{k}(\lambda,t)=0.
\end{equation}
In the same way, from (29) we obtain
\begin{equation}
b_{2n-1}^{\prime}(\lambda,t)=a^{2n}%
{\textstyle\prod\limits_{s=1}^{2n-1}}
\left(  \lambda-(2\pi(n-s)-t)^{2}\right)  ^{-1},\text{ \ }b_{k}^{\prime
}(\lambda_{n,j}(t),t)=0
\end{equation}
for\ $k<2n-1$ or $k=2m.$ Now, (30), (62) and (63) imply that the equalities
\begin{equation}
B(\lambda,t)=O\left(  n^{-5}\right)  ,\text{ }B^{\prime}(\lambda,t)=O\left(
n^{-5}\right)
\end{equation}
hold uniformly for $t\in\lbrack0,\rho]$ and $\lambda\in U(n,t,\rho).$ From
(50) and (51) (if $\left\vert u_{n,j}(t)\right\vert \geq\left\vert
v_{n,j}(t)\right\vert $ then use (50) and if $\left\vert v_{n,j}(t)\right\vert
>\left\vert u_{n,j}(t)\right\vert $ then use (51)) by using (60) and (64) we
obtain that the formula
\begin{equation}
\lambda_{n,j}(t)=(2\pi n)^{2}+O(n^{-2})
\end{equation}
holds uniformly, with respect to $t\in\lbrack0,n^{-3}],$ for $j=1,2.$

More detail estimations of $B$ and $B^{\prime}$ are given in the following
lemma, where we use the following notation. We say that $a_{n}$ is of order of
$b_{n}$ and write $a_{n}\sim b_{n}$ if $a_{n}=O(b_{n})$ and $b_{n}=O(a_{n})$
as $n\rightarrow\infty.$

\begin{lemma}
If $q$\ has the form (4), then the formulas
\begin{equation}
B(\lambda,t)=\beta_{n}\left(  1+O(n^{-2})\right)  ,\text{ }B^{\prime}%
(\lambda,t)=\alpha_{n}\left(  1+O(n^{-2})\right)  ,
\end{equation}%
\begin{equation}
\frac{\partial}{\partial\lambda}(B^{\prime}(\lambda,t)B(\lambda,t))\sim
\alpha_{n}\beta_{n}n^{-1}\ln\left\vert n\right\vert ,
\end{equation}
where $\beta_{n}=b^{2n}\left(  (2\pi)^{2n-1}(2n-1)!\right)  ^{-2}$ ,
$\alpha_{n}=a^{2n}\left(  (2\pi)^{2n-1}(2n-1)!\right)  ^{-2}$ hold uniformly
for
\begin{equation}
t\in\lbrack0,n^{-3}],\text{ }\lambda=(2\pi n)^{2}+O(n^{-2}).
\end{equation}

\end{lemma}

\begin{proof}
Using (61) and (63) by direct calculations we get
\begin{equation}
b_{2n-1}((2\pi n)^{2},0)=\beta_{n},\text{ }b_{2n-1}^{\prime}((2\pi
n)^{2},0)=\alpha_{n}.
\end{equation}
If $1\leq s\leq2n-1$ then for any $(\lambda,t)$ satisfying (68) there exists
$\lambda_{1}=(2\pi n)^{2}+O(n^{-2})$ and

$\lambda_{2}=(2\pi n)^{2}+O(n^{-2})$ such that
\begin{equation}
\mid\lambda_{1}-(2\pi(n-s))^{2}\mid<\mid\lambda-(2\pi(n-s)+t)^{2}\mid
<\mid\lambda_{2}-(2\pi(n-s))^{2}\mid.
\end{equation}
Therefore from (61) we obtain that
\begin{equation}
\left\vert b_{2n-1}(\lambda_{1},0)\right\vert <\left\vert b_{2n-1}%
(\lambda,t)\right\vert <\left\vert b_{2n-1}(\lambda_{2},0)\right\vert .
\end{equation}
On the other hand, differentiating (61) with respect to $\lambda,$ we conclude
that
\begin{equation}
\frac{\partial}{\partial\lambda}(b_{2n-1}((2\pi n)^{2},0))=b_{2n-1}((2\pi
n)^{2},0)\sum_{s=1}^{2n-1}\frac{1+O(n^{-1})}{s(2n-s)}.
\end{equation}
Now taking into account that the last summation is of order $n^{-1}%
\ln\left\vert n\right\vert $ and using (69), we get
\begin{equation}
\frac{\partial}{\partial\lambda}b_{2n-1}((2\pi n)^{2},0))\sim\beta_{n}%
n^{-1}\ln\left\vert n\right\vert .\text{ }%
\end{equation}
Arguing as above one can easily see that the $m$-th derivative, where
$m=2,3,...,$ of $b_{2n-1}(\lambda,0)$ is $O(\beta_{n}).$ Hence using the
Taylor series of $b_{2n-1}(\lambda,0)$ for $\lambda=(2\pi n)^{2}+O(n^{-2})$
about $(2\pi n)^{2},$ we obtain
\[
b_{2n-1}(\lambda_{i},0)=\beta_{n}(1+O(n^{-2})),\forall i=1,2.
\]
This with (71) yields%
\begin{equation}
b_{2n-1}(\lambda,t)=\beta_{n}(1+O(n^{-2}))
\end{equation}
for all $(\lambda,t)$ satisfying (68). In the same way, we get
\begin{equation}
\frac{\partial}{\partial\lambda}b_{2n-1}^{\prime}((2\pi n)^{2},0))\sim
\alpha_{n}(\frac{\ln n}{n}),\text{ }b_{2n-1}^{\prime}(\lambda,t)=\alpha
_{n}(1+O(n^{-2})).\text{ }%
\end{equation}

Now let us consider $b_{2n+1}(\lambda,t).$ By (57) the indices $n_{1}%
,n_{2},...,n_{2n+1}$ taking part in $b_{2n+1}(\lambda,t)$ are $\ 1$ except
one, say $n_{s+1}=-1,$ where $s=2,3,...,2n-1.$ Moreover, if $n_{s+1}=-1,$ then
$n_{1}+n_{2}+...+n_{s+1}=n_{1}+n_{2}+...+n_{s-1}=s-1$ and

$n_{1}+n_{2}+...+n_{s+2}=n_{1}+n_{2}+...+n_{s}=s.$ Therefore, by (25),
$b_{2n+1}(\lambda,t)$ for
\begin{equation}
\lambda=(2\pi n)^{2}+O(n^{-1}),\text{ }t\in\lbrack0,n^{-3}]
\end{equation}
has the form
\[
b_{2n-1}(\lambda,t)%
{\textstyle\sum\limits_{s=2}^{2n-1}}
\frac{ab}{(2\pi n)^{2}-(2\pi(n-s+1))^{2}+O(n^{-1}))(2\pi n)^{2}-(2\pi
(n-s))^{2}+O(n^{-1}))}.
\]
One can easily see that the last sum is $O(n^{-2}).$ Thus we have
\begin{equation}
b_{2n+1}(\lambda,t)=b_{2n-1}(\lambda,t)O(n^{-2})=\beta_{n}O(n^{-2}))
\end{equation}
for all $(\lambda,t)$ satisfying (76).

Now let us estimate $b_{k}(\lambda,t)$ for $k>2n+1$. Since the sums in (25)
are taken under conditions (57), we conclude that $1\leq n_{1}+n_{2}%
+\cdots+n_{s}\leq2n-1.$ Using this instead of $1\leq s\leq2n-1$ and repeating
the proof of (71) we obtain that for any $(\lambda,t)$ satisfying (76) there
exists
\[
\lambda_{3}=(2\pi n)^{2}+O(n^{-1})\text{ }\And\text{\ }\lambda_{4}=(2\pi
n)^{2}+O(n^{-1})
\]
such that%
\[
\left\vert p_{n_{1},n_{2},...,n_{k}}(\lambda_{3},0)\right\vert <\left\vert
p_{n_{1},n_{2},...,n_{k}}(\lambda,t)\right\vert <\left\vert p_{n_{1}%
,n_{2},...,n_{k}}(\lambda_{4},0)\right\vert ,\text{ }\forall k<2n-1,
\]
where $p_{n_{1},n_{2},...,n_{k}}(\lambda,0)$ is defined in (37). This with
(38) and (77) implies that
\begin{equation}
\sum_{k=2n+1}^{\infty}\left\vert b_{k}(\lambda,t)\right\vert =\beta
_{n}O(n^{-2})
\end{equation}
for all $(\lambda,t)$ satisfying (76). In the same way, we obtain
\begin{equation}
\sum_{k=2n+1}^{\infty}\left\vert b_{k}^{\prime}(\lambda,t)\right\vert
=\alpha_{n}O(n^{-2}).
\end{equation}
Thus (66) follows from (74), (75), (78) and (79).

Now we prove (67). It follows from (78), (79) and the Cauchy's inequality
that
\begin{equation}
\frac{\partial}{\partial\lambda}\left(  \sum_{k=2n+1}^{\infty}b_{k}%
(\lambda,t)\right)  =\beta_{n}O(n^{-1}),\text{ }\frac{\partial}{\partial
\lambda}\left(  \sum_{k=2n+1}^{\infty}b_{k}^{\prime}(\lambda,t)\right)
=\alpha_{n}O(n^{-1}).
\end{equation}
Therefore (67) follows from (73) and (75).
\end{proof}

From Lemma 1 it easily follows the following statement.

\begin{theorem}
If $\lambda_{n,j}(t)$ for $t\in\lbrack0,\rho]$ is a multiple eigenvalue of
$H_{t},$ then
\begin{equation}
(4\pi nt)^{2}=-\beta_{n}\alpha_{n}\left(  1+O(n^{-2})\right)  .
\end{equation}

\end{theorem}

\begin{proof}
If $\lambda_{n,1}(t)=\lambda_{n,2}(t)=:\lambda_{n}(t)$ is a multiple
eigenvalue, then as it is noted in the above, it satisfies (52) and (54) from
which we obtain%
\begin{equation}
4D(\lambda_{n}(t),t)\left(  1-\frac{1}{2}\frac{\partial}{\partial\lambda
}(A(\lambda_{n}(t),t)+A^{\prime}(\lambda_{n}(t),t))\right)  ^{2}=\left(
\frac{\partial}{\partial\lambda}D(\lambda_{n}(t),t)\right)  ^{2}.
\end{equation}
By (32) and (34) we have
\begin{equation}
\frac{\partial}{\partial\lambda}(A(\lambda_{n}(t),t)+A^{\prime}(\lambda
_{n}(t),t))=O(n^{-2}),
\end{equation}%
\begin{align}
(4\pi nt+C(\lambda_{n}(t),t))^{2}  &  =(4\pi nt)^{2}(1+O(n^{-2})),\text{ }\\
\frac{\partial}{\partial\lambda}(4\pi nt+C(\lambda_{n}(t),t))^{2}  &  =(4\pi
nt)^{2}(1+O(n^{-2}))O(n^{-3})
\end{align}
for $t\in\lbrack0,\rho]$. On the other hand, it follows from (64) and (33)
that%
\begin{equation}
B(\lambda_{n}(t),t)B^{\prime}(\lambda_{n}(t),t)=O(n^{-10}),\text{ }%
\frac{\partial}{\partial\lambda}(B^{\prime}(\lambda_{n}(t),t)B(\lambda
_{n}(t),t))=O(n^{-7}).
\end{equation}
Therefore from (53) and (84)-(86) we obtain
\begin{equation}
D((\lambda_{n}(t),t))=(4\pi nt)^{2}(1+O(n^{-2}))+O(n^{-10})
\end{equation}
and
\begin{equation}
\frac{\partial}{\partial\lambda}(D(\lambda_{n}(t),t))=(4\pi nt)^{2}%
(1+O(n^{-2}))O(n^{-3})+O(n^{-7}).
\end{equation}
Using the equalities (83), (87) and (88) in (82) we get
\begin{equation}
4(4\pi nt)^{2}(1+O(n^{-2}))=(4\pi nt)^{2}O(n^{-4})+O(n^{-8}).
\end{equation}
Hence, we have $t\in\lbrack0,n^{-3}].$ Then by (65), $t$ and $\lambda
=:\lambda_{n}(t)$ satisfy (68) and by Lemma 1
\begin{equation}
B(\lambda_{n}(t),t)=\beta_{n}\left(  1+O(n^{-2})\right)  ,\text{ }B^{\prime
}(\lambda_{n}(t),t)=\alpha_{n}\left(  1+O(n^{-2})\right)  ,
\end{equation}%
\begin{equation}
\frac{\partial}{\partial\lambda}(B^{\prime}(\lambda_{n}(t),t)B(\lambda
_{n}(t),t))\sim\alpha_{n}\beta_{n}n^{-1}\ln\left\vert n\right\vert .
\end{equation}
Therefore by (53), (84) and (85) we have
\begin{equation}
D((\lambda_{n}(t),t))=(4\pi nt)^{2}(1+O(n^{-2}))+\beta_{n}\alpha_{n}\left(
1+O(n^{-2})\right)
\end{equation}
and
\begin{equation}
\frac{\partial}{\partial\lambda}(D(\lambda_{n}(t),t))=(4\pi nt)^{2}%
(1+O(n^{-2}))O(n^{-3})+O(\alpha_{n}\beta_{n}n^{-1}\ln\left\vert n\right\vert
).
\end{equation}
Now using (83), (92) and (93) in (82) we obtain
\[
(4\pi nt)^{2}(1+O(n^{-2}))+\beta_{n}\alpha_{n}\left(  1+O(n^{-2})\right)
=(4\pi nt)^{2}O(n^{-4})+\left(  O(\alpha_{n}\beta_{n}n^{-1}\ln\left\vert
n\right\vert \right)  )^{2}%
\]
which implies (81)
\end{proof}

Note that in (92) the terms $O(n^{-2})$ don't depend on $t$, i.e, there exists
$c>0$ such that
\begin{equation}
\left\vert O(n^{-2})\right\vert <cn^{-2}%
\end{equation}
for all $t\in\lbrack0,n^{-3}].$ Henceforward, for brevity of notation,
$1+O(n^{-2})$ is denoted by $[1].$

Now we are ready to prove the main result of this section by using Theorems 6
and 7.

\begin{theorem}
Let $\mathbb{S}$ be the set of integer $n>N$ such that
\begin{equation}
-\pi+3cn^{-2}\leq\arg(\beta_{n}\alpha_{n})\leq\pi-3cn^{-2}%
\end{equation}
and $\left\{  t_{n}:n>N\right\}  $ be a sequence defined as follows: $t_{n}=0$
if $n\in$ $\mathbb{S}$ and
\begin{equation}
(4\pi nt_{n})^{2}(1-cn^{-2})=-(1+cn^{-2}+n^{-3})\operatorname{Re}(\beta
_{n}\alpha_{n})
\end{equation}
if $n\notin$ $\mathbb{S},$ where $c$ is defined in (94). Then the eigenvalues
$\lambda_{n,1}(t)$ and $\lambda_{n,2}(t)$ defined in Remark 1 are simple and
satisfy (48) for $t\in\lbrack t_{n},\rho].$
\end{theorem}

\begin{proof}
Let $n\notin$ $\mathbb{S}.$ It follows from (53), (64) and (84) that if $t\geq
n^{-3}$ then
\begin{equation}
\operatorname{Re}D(\lambda_{n,j}(t))>0.
\end{equation}
If $t\in\lbrack0,n^{-3}]$ then we have formula (92). Since the terms
$O(n^{-2})$ in (92) satisfy (94) we have the following estimate for the real
part of the first term in the right-hand side of (92):
\begin{equation}
\operatorname{Re}((4\pi nt)^{2}\left[  1\right]  )>(4\pi nt)^{2}%
(1-cn^{-2})\geq(4\pi nt_{n})^{2}(1-cn^{-2})
\end{equation}
for $t\in\lbrack t_{n},n^{-3}].$ On the other hand if $n\notin$ $\mathbb{S}$
then by the definition of $\mathbb{S}$ (95) does not hold, which implies that
\begin{equation}
\beta_{n}\alpha_{n}=-\left\vert (\beta_{n}\alpha_{n})\right\vert e^{i\theta
},\text{ }\left\vert \theta\right\vert <3cn^{-2},\text{ }\operatorname{Im}%
(\beta_{n}\alpha_{n})=O(n^{-2})\operatorname{Re}(\beta_{n}\alpha_{n}).
\end{equation}
Using this and (94), we obtain the following estimate for the real part of the
second term in the right-hand side of (92)%
\[
\left\vert \operatorname{Re}(\beta_{n}\alpha_{n}\left[  1\right]  )\right\vert
<(1+cn^{-2}+n^{-3})\left\vert \operatorname{Re}(\beta_{n}\alpha_{n}%
)\right\vert .
\]
Therefore it follows from (98), (96) and (92) that (97) holds for $t\in\lbrack
t_{n},n^{-3}]$ , $n>N$ and $n\notin$ $\mathbb{S}.$ Thus (97) is true for all
$t\in\lbrack t_{n},\rho]$. Hence $\sqrt{D(\lambda_{n,j}(t)))}$ is well-defined
and by Remark 2 it continuously depends on $t$. Therefore the proof follows
from Theorem 6.

Now consider the case $n\in$ $\mathbb{S}.$ By (94) we have
\[
-cn^{-2}-n^{-3}<\arg(\left[  1\right]  )<cn^{-2}+n^{-3}.
\]
Using (99) and (94) we obtain
\[
-\pi+2cn^{-2}-n^{-3}<\arg(\beta_{n}\alpha_{n}\left[  1\right]  )<\pi
-2cn^{-2}+n^{-3},
\]%
\[
-cn^{-2}-n^{-3}<\arg((4\pi nt)^{2}\left[  1\right]  <cn^{-2}+n^{-3}%
\]
and the acute angle between the vectors $(4\pi nt)^{2}\left[  1\right]  $ and
$\beta_{n}\alpha_{n}\left[  1\right]  $ is less than $\pi.$ Therefore by the
parallelogram law of vector addition we have
\[
-\pi<\arg(D(\lambda_{n,j}(t)))<\pi,\text{ }D(\lambda_{n,j}(t)))\neq0
\]
for $t\in\lbrack0,\rho].$ Thus the proof again follows from Theorem 6
\end{proof}

\begin{corollary}
If the relation
\begin{equation}
\text{ }\inf_{q,p\in\mathbb{N}}\{\mid2q\alpha-(2p-1)\mid\}\neq0,
\end{equation}
holds, then there exists $c_{2}\in(0,1)$ and $c_{3}\in(0,1)$ such that for all
$n>N$ the relations
\begin{align}
-\pi+c_{2}  &  <\arg(\alpha_{n}\beta_{n})<\pi-c_{2},\\
\left\vert \operatorname{Im}(\alpha_{n}\beta_{n})\right\vert  &
>c_{3}\left\vert \operatorname{Re}(\alpha_{n}\beta_{n})\right\vert
\end{align}
hold and the eigenvalues $\lambda_{n,1}(t)$ and $\lambda_{n,2}(t)$ are simple
and satisfy (48) for $t\in\lbrack0,\rho]$, where $\lambda_{n,1}(t)$,
$\lambda_{n,2}(t)$ and $N$ are defined in Remark 1.
\end{corollary}

\begin{proof}
By (100), there exists $c_{2}\in(0,1)$ such that $-\pi+c_{2}<\arg
((ab)^{2n})<\pi-c_{2}$ for all $n\in\mathbb{N}.$ Hence by the definition of
$\beta_{n}$ and $\alpha_{n}$ (see Lemma 1) (101) and hence (102) holds.
Moreover, (101) implies that (95) holds. Therefore the proof follows from
Theorem 8
\end{proof}

\begin{remark}
Let $\widetilde{A},$ $\widetilde{B},$ $\widetilde{A}^{\prime}$, $\widetilde
{B}^{\prime}$ and $\widetilde{C}$ be the functions obtained from $A,$ $B,$
$A^{\prime},$ $B^{\prime}$ and $C$ by replacing $a_{k},a_{k}^{\prime}%
,b_{k},b_{k}^{\prime}$ $\ $with $\widetilde{a}_{k},\widetilde{a}_{k}^{\prime
},\widetilde{b}_{k},\widetilde{b}_{k}^{\prime}$ , where $\widetilde{a}%
_{k},\widetilde{a}_{k}^{\prime},\widetilde{b}_{k},\widetilde{b}_{k}^{\prime}$
differ from $a_{k},a_{k}^{\prime},b_{k},b_{k}^{\prime}$ respectively, in the
following sense. The sums in the expressions for $\widetilde{a}_{k}%
,\widetilde{a}_{k}^{\prime},\widetilde{b}_{k},\widetilde{b}_{k}^{\prime}$ are
taken under condition $n_{1}+n_{2}+...+n_{s}\neq0,\pm(2n+1)$ instead of the
condition $n_{1}+n_{2}+...+n_{s}\neq0,\pm2n$ for $s=1,2,...,k.$ In
$\widetilde{b}_{k},\widetilde{b}_{k}^{\prime}$ the multiplicand $q_{\pm
2n-n_{1}-n_{2}-...-n_{k}}$ of $b_{k},b_{k}^{\prime}$ is replaced by
$q_{\pm(2n+1)-n_{1}-n_{2}-...-n_{k}}$. \ To consider the case $t\in\lbrack
\pi-\rho,\pi]$ instead of (22), (26) we use%
\begin{align*}
(\lambda_{n,j}(t)-(2\pi n+t)^{2}-\widetilde{A}(\lambda_{n,j}(t))u_{n,j}(t)  &
=(q_{2n+1}+\widetilde{B}(\lambda_{n,j}(t))v_{n,j}(t),\\
(\lambda_{n,j}(t)-(-2\pi(n+1)+t)^{2}-\widetilde{A}^{\prime}(\lambda
_{n,j}(t)))v_{n,j}(t)  &  =(q_{-2n-1}+\widetilde{B}^{\prime}(\lambda
_{n,j}(t))u_{n,j}(t)
\end{align*}
and repeat the investigations of the case $t\in\lbrack0,\rho]$. Note that
instead of (20) for $k\neq0,2n$ using the same inequality for $k\neq0,2n+1$
and $t\in\lbrack\pi-\rho,\pi]$ from (21) we obtain the last equalities instead
of (22) and (26). \ In the case $t\in\lbrack\pi-\rho,\pi]$ \ instead of (48)
we obtain
\begin{equation}
\lambda_{n,j}(t)=(2\pi n+t)^{2}-2\pi(2n+1)(t-\pi)+\frac{1}{2}(\widetilde
{A}^{^{\prime}}+\widetilde{A})+(-1)^{j}\sqrt{\widetilde{D}(\lambda_{n,j}(t))},
\end{equation}
where $\widetilde{D}=\left(  2\pi(2n+1)(t-\pi)+\widetilde{C}\right)
^{2}+\widetilde{B}$ $\widetilde{B}^{\prime}.$ Similarly, instead of (66),
(81), (96) and (100) we obtain respectively the following relations
\[
\widetilde{B}(\lambda,t)=\widetilde{\beta}_{n}\left(  1+O(n^{-2})\right)
,\text{ }\widetilde{B}^{\prime}(\lambda,t)=\widetilde{\alpha}_{n}\left(
1+O(n^{-2})\right)  ,
\]%
\[
\left(  2\pi(2n+1)(t-\pi)\right)  ^{2}=-\widetilde{\beta}_{n}\widetilde
{\alpha}_{n}\left(  1+O(n^{-2})\right)  ,
\]%
\[
(2\pi(2n+1)(\widetilde{t}_{n}-\pi))^{2}(1-cn^{-2})=-(1+cn^{-2}+n^{-3}%
)\operatorname{Re}(\widetilde{\beta}_{n}\widetilde{\alpha}_{n}),
\]%
\begin{equation}
\text{ }\inf_{q,p\in\mathbb{N}}\{\mid(2q+1)\alpha-(2p-1)\mid\}\neq0,
\end{equation}
where $\widetilde{\beta}_{n}=b^{2n+1}\left(  (2\pi)^{2n}(2n)!\right)  ^{-2}$,
$\widetilde{\alpha}_{n}=a^{2n+1}\left(  (2\pi)^{2n}(2n)!\right)  ^{-2},$
$\widetilde{t}_{n}\in\lbrack\pi-\rho,\pi]$ and Theorems 7 and 8 and Corollary
2 continue to hold under the corresponding replacement.

As we noted in Section 2 (see \textit{Theorem 2 of [12] and Remark 1}) the
large eigenvalues of $H_{t}$ for $t\in\lbrack\rho,\pi-\rho]$ consist of the
simple eigenvalues $\lambda_{n}(t)$ for $\left\vert n\right\vert >N$
satisfying the, \textit{uniform with respect to }$t$\textit{ in }$[\rho
,\pi-\rho],$\textit{ asymptotic formula (17).} Thus by Theorem 8 and by the
just noted similar investigation, the eigenvalues $\lambda_{n,j}(t)$ for
$n>N,$ $j=1,2$ and $t\in$ $\left(  [t_{n},\rho]\cup\lbrack\pi-\rho
,\widetilde{t}_{n}]\right)  $ and the eigenvalues $\lambda_{n}(t)$ \textit{for
\ }$t\in\lbrack\rho,\pi-\rho]$\textit{ and }$\left\vert n\right\vert
>N$\textit{ are simple.}. These eigenvalues satisfy (48), (17) and (103) for
$t\in$ $[t_{n},\rho],$ $t\in\lbrack\rho,\pi-\rho]$ and $t\in\lbrack\pi
-\rho,\widetilde{t}_{n}]$ respectively. Finally, note that (100) and (104)
hold if and only if (8) holds.
\end{remark}

\section{On the Spectrality of H}

In this section we find necessary and sufficient condition on $a$ and $b$ for
the asymptotic spectrality of the operator $H(a,b)$, that is, we prove Theorem
1 formulated in the introduction. To prove this main result of this section we
first prove the following two statements which easily follows from the results
of Section 3.

\begin{theorem}
If (8) holds, then there exists $N$ such that for $\left\vert n\right\vert >N$
the component $\Gamma_{n}$ of the spectrum $S(H)$ of the operator $H$ is a
separated simple analytic arc with the end points $\lambda_{n}(0)$ and
$\lambda_{n}(\pi).$ These components do not contain spectral singularities. In
other words, the number of the spectral singularities of $H$ is finite.
\end{theorem}

\begin{proof}
As we noted in the end of Remark 3 if (8) holds, then (100) and (104) hold
too. Therefore by Corollary 2, Theorem 2 of [13]\textit{,} and Remark 3 the
eigenvalues $\lambda_{n}(t)$ for $\left\vert n\right\vert >N$ and $t\in
\lbrack0,\pi]$ are simple. Therefore for $\left\vert n\right\vert >N$ the
component $\Gamma_{n}$ of the spectrum of the operator $H$ is a separated
simple analytic arc with the end points $\lambda_{n}(0)$ and $\lambda_{n}%
(\pi)$. It is well-known that the spectral singularities of $H$ are contained
in the set of multiple eigenvalues of $H_{t}$ (see Proposition 2 of [17]).
Hence, $\Gamma_{n}$ for $\left\vert n\right\vert >N$ does not contain the
spectral singularities. On the other hand, the multiple eigenvalues are the
zeros of the entire function $\frac{dF(\lambda)}{d\lambda},$ where
$F(\lambda)$ is defined in (3). Since the entire function has a finite number
of roots on the bounded sets the number of the spectral singularities of
$H(a,b)$ is finite
\end{proof}

It was noted in [2] that (see page 539 of [2]) if$\ |a|\neq|b$%
$\vert$%
, then the results of [6] and [2] show that $H(a,b)$ is not a spectral
operator. Since our aim is to prove the necessary and sufficient condition for
asymptotic spectrality and the fact that $H(a,b)$ is not a spectral operator
does not imply that it is not asymptotic spectral operator, here we prove the
following fact which easily follows from the formulas of Section 3.

\begin{proposition}
If $\mid a\mid\neq\mid b\mid,$ then the operator $H(a,b)$ has the spectral
singularity at infinity and hence is not an asymptotically spectral operator.
\end{proposition}

\begin{proof}
Suppose, without loss of generality, that $\mid a\mid<\mid b\mid.$ By Theorem
7 large periodic eigenvalues $\lambda_{n}(0)$ are simple. Due to (49),
formulas (22), (26) and (36) for $t=0$ have the forms
\begin{align}
(\lambda_{n}(0)-(2\pi n)^{2}-A(\lambda_{n}(0),0))u_{n}  &  =B(\lambda
_{n}(0),0))v_{n},\\
(\lambda_{n}(0)-(2\pi n)^{2}-A^{\prime}(\lambda_{n}(0),0))v_{n}  &
=B^{\prime}(\lambda_{n}(0),0))u_{n},\text{ }\\
\left\vert u_{n}\right\vert ^{2}+\left\vert v_{n}\right\vert ^{2}  &
=1+O(n^{-2}),
\end{align}
where $u_{n}=(\Psi_{n,0},e^{i2\pi nx}),$ $v_{n}=(\Psi_{n,0},e^{-i2\pi nx})$
and $\lambda_{n,j}(0)$ is redenoted by $\lambda_{n}(0).$ By (66),
$B(\lambda_{n}(0),0))$ and $B^{\prime}(\lambda_{n}(0),0))$ are nonzero
numbers. Moreover, by Lemma 3 of [11] we have $A(\lambda_{n}(0),0))=A^{\prime
}(\lambda_{n}(0),0)).$ Therefore equalities (105)-(107) imply that
\[
(\lambda_{n}(0)-(2\pi n)^{2}-A(\lambda_{n}(0),0))=(\lambda_{n}(0)-(2\pi
n)^{2}-A(\lambda_{n}(0),0))\neq0,
\]
$u_{n}\neq0$ and $v_{n}\neq0.$ Thus, dividing (105) and (106) side by side and
using (66) we get
\begin{equation}
\frac{u_{n}^{2}(0)}{v_{n}^{2}(0)}=\frac{B(\lambda_{n}(0),0))}{B^{\prime
}(\lambda_{n}(0),0))}=O\left(  \frac{\left\vert a\right\vert ^{n}}{\left\vert
b\right\vert ^{n}}\right)  =O(n^{-2}).
\end{equation}
Using this equality and (35) we obtain
\[
u_{n}=v_{n}O(n^{-1})=O(n^{-1}),\text{ }\Psi_{n,0}(x)=c_{4}e^{-i2\pi
nx}+O(n^{-1}),
\]
where $\mid c_{4}\mid=1$ and $\Psi_{n,0}(x)$ is the normalized eigenfunction
corresponding to $\lambda_{n}(0).$ Replacing $a$ and $b$ by $\overline{b}$ and
$\overline{a}$ respectively, in the same way we obtain
\[
\Psi_{n,0}^{\ast}(x)=c_{5}e^{i2\pi nx}+O(n^{-1}),\text{ }\left\vert
c_{5}\right\vert =1.
\]
Thus $(\Psi_{n,0},\Psi_{n,0}^{\ast}(x))\rightarrow0$ as $n\rightarrow\infty$
and hence the proof follows from Definitions 2 and 3
\end{proof}

Thus the last theorem shows that if $\mid a\mid\neq\mid b\mid,$ then $H(a,b)$
is not an asymptotically spectral operator and for this in is enough to
consider the case $t=0.$ However the inverse statement is not true. Moreover,
by (6) and Definition 2, to find the condition for asymptotic spectrality we
need to consider $d_{n}(t)=(\Psi_{n,t},\Psi_{n,t}^{\ast}(x))$ and get an
estimation (7) for large $n$ and for all values of $t\in(-\pi,\pi],$ when
$\lambda_{n}(t)$ is a simple eigenvalue. The proof of (7) for $t\in\lbrack
\rho,\pi-\rho]$ follows from (17). Now we estimate $\left\vert d_{n}%
(t)\right\vert ^{-1}$ for $t\in\lbrack0,\rho],$ which is the main difficulty
of this section. Especially, it is very hart to estimate it when $(4\pi
nt)^{2}$ lies in the neighborhood of $-\beta_{n}\alpha_{n},$ since in this
case $\lambda_{n}(t)$ may became a multiple eigenvalue due to Theorem 7. The
estimation for $t\in\lbrack\pi-\rho,\pi]$ is similar.

\begin{remark}
Henceforward, for brevity of notation and according to Remark 1 instead of
$\lambda_{n,2}(t),$ $u_{n,2}(t)$ and $v_{n,2}(t)$ we use the notation
$\lambda_{n}(t),$ $u_{n}(t)$ and $v_{n}(t)$ and consider the case $n>N.$ The
case $n<-N$ is similar. Moreover, we redenote the numbers $C(\lambda
_{n,2}(t),t),$ $D(\lambda_{n,2}(t),t),$ $B(\lambda_{n,2}(t),t)$, and
$B^{\prime}(\lambda_{n,2}(t),t)$ by $C(\lambda_{n}(t)),$ $D(\lambda_{n}(t)),$
$B(\lambda_{n}(t))$, $B^{\prime}(\lambda_{n}(t)).$ By (84), Lemma 1 and (53)
the equalities%
\begin{align*}
4\pi nt+C(\lambda_{n}(t))  &  =(4\pi nt)\left[  1\right]  ,\text{ }%
B(\lambda_{n}(t))=\beta_{n}\left[  1\right]  ,\text{ }B^{\prime}(\lambda
_{n}(t))=\alpha_{n}\left[  1\right]  ,\\
D(\lambda_{n}(t))  &  =(4\pi nt)^{2}[1]+\alpha_{n}\beta_{n}[1]
\end{align*}
hold uniformly for $t\in\lbrack0,n^{-3}],$ where $[1]=1+O(n^{-2}).$ By Theorem
8, $\lambda_{n}(t)$ satisfies (48) for $j=2$ \ and $t\in\lbrack t_{n},\rho].$
If $t\in\lbrack0,t_{n}),$ then $\lambda_{n}(t)$ satisfies either (44) or (45).
\end{remark}

Using formula (48) for $j=2$ in (50) and (51) and taking into account the
notations and arguments of Remark 4 we obtain
\begin{equation}
E_{-}(\lambda_{n}(t))u_{n}(t)=\beta_{n}v_{n}(t)[1],
\end{equation}%
\begin{equation}
E_{+}(\lambda_{n}(t))v_{n}(t)=\alpha_{n}u_{n}(t)[1],
\end{equation}
where $[1]=1+O(n^{-2})$,%
\begin{equation}
E_{\pm}(\lambda_{n}(t))=s(t)\sqrt{(4\pi nt)^{2}[1]+\alpha_{n}\beta_{n}[1]}%
\pm4\pi nt[1],
\end{equation}
$4\pi nt[1]=C(\lambda_{n}(t))+4\pi nt,$ $s(t)$ is $-1$ or $1$ if $\lambda
_{n}(t)$ satisfies (44) or (45) respectively.

Since the boundary condition (2) is self-adjoint we have $(H_{t}(q))^{\ast}=$
$H_{t}(\overline{q}).$ Therefore, all formulas and theorems obtained for
$H_{t}$ are true for $H_{t}^{\ast}$ if we replace $a$ and $b$ by $\overline
{b}$ and $\overline{a}$ respectively. For instance, (35) and (36) hold for the
operator $H_{t}^{\ast}$ and hence we have
\begin{equation}
\Psi_{n,t}^{\ast}(x)=u_{n}^{\ast}(t)e^{i(2\pi n+t)x}+v_{n}^{\ast}(t)e^{i(-2\pi
n+t)x}+h_{n,t}^{\ast}(x),
\end{equation}%
\begin{equation}
(h_{n,t}^{\ast},e^{i(\pm2\pi n+t)x})=0,\text{ }\left\Vert h_{n,t}^{\ast
}\right\Vert =O(n^{-1}),\text{ }\left\vert u_{n}^{\ast}(t)\right\vert
^{2}+\left\vert v_{n}^{\ast}(t)\right\vert ^{2}=1+O(n^{-2}).
\end{equation}
Similarly the formulas (109) and (110) for the operator $H_{t}^{\ast}$ have
the form%
\begin{equation}
\overline{E_{-}(\lambda_{n}(t))}u_{n}^{\ast}(t)=\overline{\alpha_{n}}%
v_{n}^{\ast}(t)[1],\text{ }\overline{E_{+}(\lambda_{n}(t))}v_{n}^{\ast
}(t)=\overline{\beta_{n}}u_{n}^{\ast}(t)[1].
\end{equation}
For $\mid n\mid>N$ it follows from (35), (36), (112) and (113) that%
\begin{equation}
(\Psi_{n,t},\Psi_{n,t}^{\ast})=u_{n}(t)\overline{u_{n}^{\ast}(t)}%
+v_{n}(t)\overline{v_{n}^{\ast}(t)}+O(n^{-1}).
\end{equation}
By (6) and Definition 2 to study the asymptotic spectrality we need to
consider the expression $(\Psi_{n,t},\Psi_{n,t}^{\ast})$. First let us note
the following simple statement.

\begin{proposition}
The equality $(\Psi_{n,t},\Psi_{n,t}^{\ast})=1+O(n^{-1})$ holds uniformly for
$t\in\lbrack n^{-3},\rho].$
\end{proposition}

\begin{proof}
If $t\in\lbrack n^{-3},\rho],$ then by (84), (53) and (64) the coefficient of
$v_{n}(t)$ in (110) is greater than $n$ times of the coefficient of $u_{n}%
(t)$. Therefore from (35) and (36) we get
\[
\Psi_{n,t}(x)=e^{i(2\pi n+t)x}+O(n^{-1}).
\]
Instead of (110), (35) and (36) using (114), (112) and (113) in the same way
we obtain that $\Psi_{n,t}^{\ast}$ satisfies the same formula. These formulas
imply the proof of the proposition.
\end{proof}

By (115) to estimate $(\Psi_{n,t},\Psi_{n,t}^{\ast})$ we need to consider
$u_{n}(t)\overline{u_{n}^{\ast}(t)}+v_{n}(t)\overline{v_{n}^{\ast}(t)}.$ Using
(110), (114) and the obvious equality
\begin{equation}
E_{+}(\lambda_{n}(t))E_{-}(\lambda_{n}(t))=\alpha_{n}\beta_{n}[1]
\end{equation}
we obtain $v_{n}(t)\overline{v_{n}^{\ast}}(t)\left(  E_{+}(\lambda
_{n}(t))\right)  ^{2}=u_{n}(t)\overline{u_{n}^{\ast}(t)}\alpha_{n}\beta
_{n}[1]$ and%
\begin{equation}
\frac{v_{n}(t)\overline{v_{n}^{\ast}(t)}}{u_{n}(t)\overline{u_{n}^{\ast}(t)}%
}=\frac{\alpha_{n}\beta_{n}[1]}{\left(  E_{+}(\lambda_{n}(t))\right)  ^{2}%
}=\frac{E_{-}(\lambda_{n}(t))[1]}{E_{+}(\lambda_{n}(t))}.
\end{equation}
This with definition of $E_{\pm}(\lambda_{n}(t))$ implies that
\begin{equation}
u_{n}(t)\overline{u_{n}^{\ast}}(t)+v_{n}(t)\overline{v_{n}^{\ast}}%
(t)=u_{n}(t)\overline{u_{n}^{\ast}}(t)F_{+}(\lambda_{n}(t)),
\end{equation}
where
\[
F_{+}(\lambda_{n}(t))=\frac{G(\lambda_{n}(t))}{E_{+}(\lambda_{n}(t))},\text{
}G(\lambda_{n}(t))=O(n^{-2})E_{-}(\lambda_{n}(t))+2s(t)\sqrt{(4\pi
nt)^{2}[1]+\alpha_{n}\beta_{n}[1]}.
\]

Therefore in the following lemma we investigate $F_{+}(\lambda_{n}(t))$ and
$u_{n}(t)\overline{u_{n}^{\ast}}(t)$.

\begin{remark}
By Proposition 2 we need to estimate $(\Psi_{n,t},\Psi_{n,t}^{\ast})$ for
$t\in\lbrack0,n^{-3}].$ For this we divide the last interval into three
subintervals $I_{1},I_{2}$ \ and $I_{3}$ , where $I_{k}$ for $k=1,2,3$ are
respectively the set of all $t\in\lbrack0,n^{-3}]$ such that $4\pi nt$ belongs
to the sets $\left[  0,\frac{1}{4}\varepsilon_{n}\right]  ,$ $\left(  \frac
{1}{4}\varepsilon_{n},\frac{5}{4}\varepsilon_{n}\right)  $ and $\left[
\frac{5}{4}\varepsilon_{n},4\pi n^{-2}\right]  ,$ where $\varepsilon_{n}%
=\sqrt{\left\vert \alpha_{n}\beta_{n}\right\vert }.$ It follows from (96) that
$I_{3}\subset\lbrack t_{n},n^{-3}].$ Therefore $s(t)=1$ if $t\in I_{3}.$
\end{remark}

\begin{lemma}
$(a)$ The relation
\begin{equation}
F_{+}(\lambda_{n}(t))\sim1
\end{equation}
holds uniformly for $t\in\left(  I_{1}\cup I_{3}\right)  $. If $\mid
a\mid=\mid b\mid,$ then there exists $c_{6}>0$ such that
\begin{equation}
\left\vert u_{n}(t)\overline{u_{n}^{\ast}}(t)\right\vert >c_{6}%
\end{equation}
for all $t\in\left(  I_{1}\cup I_{3}\right)  .$

$(b)$ If (8) holds then the relation (119) hold uniformly for $t\in I_{2}.$ If
$\mid a\mid=\mid b\mid,$ then (120) holds for all $t\in I_{2}.$
\end{lemma}

\begin{proof}
$(a)$\ If $t\in I_{1}$ then we have $4\pi nt\leq\frac{1}{4}\varepsilon_{n}.$
It implies that
\begin{equation}
G(\lambda_{n}(t))\sim\varepsilon_{n},\text{ }E_{+}(\lambda_{n}(t))\sim
\varepsilon_{n}%
\end{equation}
and hence (119) holds. If $t\in I_{3},$ then $s(t)=1$ (see Remark 5) and by
(46) we have
\begin{equation}
G(\lambda_{n}(t))\sim4\pi nt,\text{ }E_{+}(\lambda_{n}(t))\sim4\pi nt.
\end{equation}
Therefore (119) holds. Now suppose that $\mid a\mid=\mid b\mid.$ If $t\in
I_{1}$, then using (121) and taking into account that$\mid\alpha_{n}\mid
=\mid\beta_{n}\mid$ when $\mid a\mid=\mid b\mid,$ we obtain $E_{+}(\lambda
_{n}(t))\sim\alpha_{n}.$ Therefore from (110) and (36) we obtain $u_{n}(t)\sim
v_{n}(t)$ $\sim1.$ In the same way from (114) and (113) we get $u_{n}^{\ast
}(t)\sim v_{n}^{\ast}(t)$ $\sim1.$ These relations imply (120). If $t\in
I_{3}$ and $\mid a\mid=\mid b\mid,$ then using (46) we see that $\left\vert
E_{+}(\lambda_{n}(t))\right\vert >\left\vert \alpha_{n}\right\vert $.
Therefore using (110) and (36) we obtain $\mid u_{n}(t)\mid>2/3.$ Similarly
$\mid u_{n}^{\ast}(t)\mid>2/3.$ Thus (120) holds.

$(b)$ If (8) holds then we have inequality (102). Using it and the relation
$t\in I_{2}$ we obtain%
\begin{equation}
\operatorname{Im}((4\pi nt)^{2}[1]+\alpha_{n}\beta_{n}[1])\sim\alpha_{n}%
\beta_{n}\sim(4\pi nt)^{2},\text{ }E_{\pm}(\lambda_{n}(t))\sim G(\lambda
_{n}(t))\sim\varepsilon_{n}.
\end{equation}
It implies (119). If $\mid a\mid=\mid b\mid$ and $t\in I_{2},$ then
$\mid\alpha_{n}\mid=\mid\beta_{n}\mid$ and $4\pi nt\sim\alpha_{n}.$ Therefore
one can easily verify that
\begin{equation}
E_{+}(\lambda_{n}(t))\sim\alpha_{n}.
\end{equation}
Using it and arguing as in the case $t\in I_{1}$ we get the proof of (120).
\end{proof}

Now we prove the main result (extended version of Theorem 1) of this section.

\begin{theorem}
$(a)$ The operator $H$ has no the spectral singularity at infinity and is an
asymptotically spectral operator if and only if $\mid a\mid=\mid b\mid$ and
(8) holds.

$(b)$ Let $\mid a\mid=\mid b\mid.$ If $\alpha$ is a rational number, that is,
$\alpha=\frac{m}{q}$ where $m$ and $q$ are irreducible integers and $\alpha$
is defined in (8), then the operator $H$ has no the spectral singularity at
infinity and is an asymptotically spectral operator if and only if $m$ is an
even integer. If $\alpha$ is an irrational number, then $H$ has the spectral
singularity at infinity and is not an asymptotically spectral operator if and
only if there exists a sequence of pairs $\{(q_{k},p_{k})\}\subset
\mathbb{N}^{2}$ such that$\ $%
\[
\mid\alpha-\left(  2p_{k}-1\right)  \left(  q_{k}\right)  ^{-1}\mid=o\left(
\left(  q_{k}\right)  ^{-1}\right)  ,
\]
where $2p_{k}-1$ and $q_{k}$ are irreducible integers.
\end{theorem}

\begin{proof}
It is clear that $(b)$ follows from $(a).$ First we prove the sufficiency of
$(a)$. For this assuming that $\mid a\mid=\mid b\mid$ and (8) holds, we prove
(7) for large $n$. If $t\in\lbrack n^{-3},\rho],$ then by Proposition 2, (7)
holds. Using (115), (118), and Lemma 2 we get (7) in the case $t\in
\lbrack0,n^{-3}]$. Hence (7) for $t\in\lbrack0,\rho]$ is proved. In the same
way, by using Remark 3, we prove (7) for $t\in\lbrack\pi-\rho,\pi].$ If
$t\in\lbrack\rho,\pi-\rho],$ then (7) follows from (17). Thus (7) folds for
$t\in\lbrack0,\pi]$ and $n>N.$ In the same way we prove it for $t\in(-\pi,0)$
and $n<-N$.

It remains to prove the necessity of $(a)$. Suppose that $H(a,b)$ is an
asymptotically spectral operator. Then by Proposition 1, $\mid a\mid=\mid
b\mid.$ Now we prove that (8) holds. Suppose to the contrary that (8) does not
hold. Then there exists a sequence of pairs $\{(q_{k},p_{k})\}$ such that
$q_{k}\alpha-(2p_{k}-1)\rightarrow0.$ First suppose that the sequence
$\{q_{k}\}$ contains infinite many of even number. Then one can easily verify
that there exists a sequence $\{n_{k}\}$ satisfying
\[
\operatorname{Im}((ab)^{2n_{k}})=o((ab)^{2n_{k}}),\text{ }\lim_{k\rightarrow
\infty}sgn(\operatorname{Re}((ab)^{2n_{k}}))=-1.
\]
By Theorem 8, for the sequence $\{t_{n_{k}}\}$ defined by (96) and now, for
simplicity, redenoted by $\{t_{k}\}$ the eigenvalues $\lambda_{n_{k},j}%
(t_{k})$ are simple and the following relations hold%
\[
(4\pi n_{k}t_{k})^{2}=-\operatorname{Re}(\beta_{n_{k}}\alpha_{n_{k}%
})(1+o(1))=-(\beta_{n_{k}}\alpha_{n_{k}})(1+o(1)),\text{ }%
\]%
\[
(4\pi n_{k}t_{k})^{2}+(\beta_{n_{k}}\alpha_{n_{k}})=o(\beta_{n_{k}}%
^{2}),\text{ }4\pi n_{k}t_{k}\sim\beta_{n_{k}}\sim\alpha_{n_{k}}.
\]
Therefore we have $G(\lambda_{n_{k}}(t_{k}))=o(\beta_{n_{k}}).$ It with (124)
implies that $F_{+}(\lambda_{n_{k}}(t_{k}))=o(1)$ as $k\rightarrow\infty.$
Thus $\left\vert d_{n_{k}}(t_{k})\right\vert \rightarrow0$ as $k\rightarrow
\infty$, due to (115) and (118). In the same way we prove it when $\{q_{k}\}$
contains infinite number of odd number. It contradicts to the assumption that
$H(a,b)$ is an asymptotically spectral operator, due to Definition 2
\end{proof}

Now using Theorem 1, that is, Theorem 10 (a) we prove Corollary 1 (see introduction).

\textbf{The proof of Corollary 1.} Since any self-adjoint operator is
spectral, we need to prove that if $H(a,b)$ is a spectral operator and $ab$ is
real, then (4) is a real potential. By Definitions 1 and 2 the spectral
operator is also asymptotically spectral operator. Thus $ab$ is real and by
Theorem 1, (8) holds. If $ab<0$ then $\alpha-1=0,$where $\alpha=\pi^{-1}%
\arg(ab)$, which contradicts (8). Hence we have $ab>0.$ On the other hand,
Proposition 1 implies that $\mid a\mid=\mid b\mid.$ From the last two
relations we obtain $b=\overline{a}$ . It means that (4) is a real potential
and $H(a,b)$ is a self-adjoint operator.

\section{On the Spectral Expansion of $H(a,b)$}

Now we consider the forms of the spectral expansion of $H(a,b)$. For this as
is noted in the introduction we need to investigate in detail the ESS and ESS
at infinity for $H(a,b)$. Besides, we use the following results of the papers
[14, 17, 18] formulated as summary.

\begin{summary}
$(a)$ The spectral expansion has the elegant form (9) if and only if $L(q)$
has no ESS and ESS at infinity (see page 7 of [18]).

$(b)$ If $L(q)$ has no ESS at infinity, then the number of ESS is at most
finite and the spectral expansion has the asymptotically elegant form (13)
(see Theorem 3.13 of [18]).

$(c)$ ESS of $L(q)$ is a multiple 2-periodic eigenvalue. Note that the
eigenvalues of $L_{0}(q)$ and $L_{\pi}(q)$ is called as 2-periodic
eigenvalues. If the geometric multiplicity of the multiple 2-periodic
eigenvalue is 1, then it is ESS (see Proposition 4 of [17]).

$(d)$ If $0<\left\vert ab\right\vert <16/9$ then all 2-periodic eigenvalues of
$H(a,b)$ are simple (see Theorems 13 and 15 of [14]).

$(e)$ If \ $\Lambda=\lambda_{n}(t_{0})$ is a multiple eigenvalue, then the sum
of the expressions $a_{k}(t)\Psi_{k,t}(x)$ for $k\in\left\{  s\in
\mathbb{Z}:\lambda_{s}(t_{0})=\Lambda\right\}  $ is integrable in some
neighborhood of $t_{0}$. If $\Lambda$ is an ESS then at least two of these
expressions are nonintegrable (see Remark 2 of [17]).
\end{summary}

As we noted at the end of introduction, in this section we consider the
spectral expansion of $H(a,b)$ for all potential of the form (4) by dividing
it into two complementary cases:

Case 1: $ab\neq0$ and Case 2: either $a=0$ or $b=0.$

First we consider Case 1 and prove that in this case the operator $H(a,b)$ has
no ESS at infinity. Therefore by Summary 1$(b)$ the number of ESS is at most
finite and the spectral expansion has the asymptotically elegant form (13).
For this, due to Definition 5, we need to study the existence and the behavior
of
\begin{equation}
\int_{(-\pi,\pi]}\left\vert d_{n}(t)\right\vert ^{-1}dt
\end{equation}
for large $n$. Note that the estimations that was done in Section 4 for
$\left\vert d_{n}(t)\right\vert ^{-1}$ are not enough for the estimations of
(125). Here we need more sharp and complicated estimations due to the
followings. In Section 4 some estimations for $(\Psi_{n,t},\Psi_{n,t}^{\ast})$
were done under assumption $\mid a\mid=\mid b\mid$ (first condition) and the
estimations for $t\in I_{2},$ where a multiple eigenvalues may appear, were
done under condition (8) (second condition), while in this section the
estimations are done for all cases of the potential (4). Moreover in Section 4
we considered only boundlessness of \ $\left\vert d_{n}(t)\right\vert ^{-1},$
while here we consider its integrability and investigate the limit of (125) as
$n\rightarrow\infty.$

In this section to estimate $(\Psi_{n,t},\Psi_{n,t}^{\ast})$ we also use
(118). However, if the first condition does not hold, say if $\mid a\mid<\mid
b\mid,$ then the multiplicand $u_{n}(t)\overline{u_{n}^{\ast}}(t)$ and hence
the left-hand side of (118) is $O\left(  \mid a\mid^{n}\mid b\mid^{-n}\right)
.$ Therefore (115) is ineffective for the estimation of $(\Psi_{n,t}%
,\Psi_{n,t}^{\ast}).$ For this first of all we consider the Bloch functions in
detail which was done in Theorem 11. Moreover, in Section 4 we have used
essentially the second condition (8) to estimate $(\Psi_{n,t},\Psi_{n,t}%
^{\ast})$ for $t\in I_{2}.$ To do this estimation without condition (8) we
develop a new approach at the end in this section. Besides we use the first
statements of Lemma 2(a) and Lemma 2(b) which hold without assumption $\mid
a\mid=\mid b\mid.$ Finally, note that in the proof of Lemma 2(b) we proved
that if (102) holds then (119) holds uniformly for $t\in I_{2}.$

\begin{theorem}
If $\mid a\mid<\mid b\mid$ then%
\begin{equation}
(\Psi_{n,t},\Psi_{n,t}^{\ast})=u_{n}(t)\overline{u_{n}^{\ast}(t)}%
[1]+v_{n}(t)\overline{v_{n}^{\ast}(t)}[1]+O\left(  n^{-1}a^{n}b^{-n}\right)  .
\end{equation}

\end{theorem}

\begin{proof}
First we write the function $h_{n,2,t}$ defined in (35) as sum of $f_{n}$ and
$g_{n}$ defined by
\[
f_{n}(x)=%
{\textstyle\sum\limits_{0<\left\vert k\right\vert \leq s}}
\left(  \left(  \Psi_{n,t},e^{i(2\pi\left(  n+k\right)  +t)x}\right)
e^{i(2\pi\left(  n+k\right)  +t)x}+\left(  \Psi_{n,t},e^{i(2\pi\left(
-n+k\right)  +t)x}\right)  e^{i(2\pi\left(  -n+k\right)  +t)x}\right)
\]
and
\[
g_{n}(x)=%
{\textstyle\sum\limits_{k:\left\vert k\pm n\right\vert >s}}
\left(  \Psi_{n,t},e^{i(2\pi k+t)x}\right)  e^{i(2\pi k+t)x},
\]
where $s$ is a positive integer of order $n/\ln n$ such that
\begin{equation}
n^{-s}=O\left(  n^{-1}a^{n}b^{-n}\left(  \left\vert a\right\vert +\left\vert
b\right\vert \right)  ^{-s}\right)  .
\end{equation}
If $\left\vert k\pm n\right\vert >s$ then iterating the formula%
\begin{equation}
(\Psi_{n,t},e^{i(2\pi k+t)x})=\frac{a(\Psi_{n,t},e^{i(2\pi\left(  k+1\right)
+t)x})+b(\Psi_{n,t},e^{i(2\pi\left(  k-1\right)  +t)x})}{(\lambda_{n}(t)-(2\pi
k+t)^{2})}%
\end{equation}
$s$ times we obtain%
\[
(\Psi_{n,t},e^{i(2\pi k+t)x})=\sum_{n_{1},n_{2},...,n_{s}}\frac{q_{n_{1}%
}q_{n_{2}}...q_{n_{s}}(\Psi_{n,t},e^{i(2\pi(k-n_{1}-n_{2}-...-n_{s})+t)x})}{%
{\textstyle\prod\limits_{j=1,2,...,s}}
[\lambda_{n}(t)-(2\pi(k-n_{1}-n_{2}-...-n_{j})+t)^{2}]},
\]
where $n_{j}$ is either $1$ or $-1$ and $q_{-1}=a,$ $q_{1}=b.$ Therefore using
(20) and (127) we obtain
\[
\left\vert (\Psi_{n,t},e^{i(2\pi k+t)x})\right\vert \leq\frac{\left(
\left\vert a\right\vert +\left\vert b\right\vert \right)  ^{s}}{%
{\textstyle\prod\limits_{j=1,2,...,s}}
\left\vert n^{2}-(n+k-j)^{2}\right\vert }=O\left(  \frac{\left(  \left\vert
a\right\vert +\left\vert b\right\vert \right)  ^{s}}{kn^{s}}\right)  =O\left(
\frac{a^{n}}{knb^{n}}\right)  .
\]
If $0<\left\vert k\right\vert \leq s$ then we iterate the formula
\[
(\Psi_{n,t},e^{i(2\pi\left(  n+k\right)  +t)x})=\frac{a(\Psi_{n,t}%
,e^{i(2\pi\left(  \left(  n+k\right)  +1\right)  +t)x})+b(\Psi_{n,t}%
,e^{i(2\pi\left(  \left(  n+k\right)  -1\right)  +t)x})}{(\lambda_{n}%
(t)-(2\pi\left(  n+k\right)  +t)^{2})}%
\]
obtained from (128) by replacing $k$ with $n+k$ as follows. After each
iteration we isolate the term containing $(\Psi_{n,t}(x),e^{i(2\pi n+t)x})$
and iterate the other terms. \ Continuing these procedure $s$ times and
estimating as above we obtain
\[
(\Psi_{n,t},e^{i(2\pi\left(  n+k\right)  +t)x})=u_{n}(t)O(n^{-2})+O\left(
n^{-1}k^{-1}a^{n}b^{-n}\right)  .
\]
In the same way we obtain
\[
(\Psi_{n,t},e^{i(2\pi\left(  -n+k\right)  +t)x})=v_{n}(t)O(n^{-2})+O\left(
n^{-1}k^{-1}a^{n}b^{-n}\right)  .
\]
The same estimations holds for the eigenfunction $\Psi_{n,t}^{\ast}.$ These
estimations imply (126).
\end{proof}

Now to estimate $\left\vert d_{n}(t)\right\vert $ we use (126) and the
following formula
\begin{equation}
u_{n}\overline{u_{n}^{\ast}}[1]+v_{n}\overline{v_{n}^{\ast}}[1]=[1]u_{n}%
\overline{u_{n}^{\ast}}F_{+}(\lambda_{n}(t)),
\end{equation}
where $F_{+}(\lambda_{n}(t))$ is defined in (118) and the proof of (129) can
be obtained by repeating the proofs of (118). In Lemma 2 the expression
$u_{n}\overline{u_{n}^{\ast}}$ is estimated under condition $\left\vert
a\right\vert =\left\vert b\right\vert .$ Now we estimate it if this condition
does not holds and without loss of generality assume that $\left\vert
a\right\vert <\left\vert b\right\vert .$ First we estimate $u_{n}.$

\begin{lemma}
If $\left\vert a\right\vert <\left\vert b\right\vert $ and $t\in\lbrack
0,\rho],$ then
\begin{equation}
u_{n}(t)=1+O(n^{-1}).
\end{equation}

\end{lemma}

\begin{proof}
First study the case $4\pi nt\in\left[  0,2\varepsilon_{n}\right]  ,$ where
$\varepsilon_{n}$ is defined in Remark 5. Then by (111) $\left\vert
E_{-}(\lambda_{n}(t))\right\vert <5\varepsilon_{n}.$ Therefore, using (109)
and then taking into account that%
\begin{equation}
\left\vert \alpha_{n}\right\vert \left\vert \beta_{n}\right\vert
^{-1}=\left\vert a^{2n}\right\vert \left\vert b^{2n}\right\vert ^{-1}=O\left(
n^{-2}\right)
\end{equation}
(see Lemma 1) we obtain
\[
\left\vert v_{n}(t)\right\vert =\frac{\left\vert E_{-}(\lambda_{n}%
(t))\right\vert }{\left\vert \beta_{n}[1]\right\vert }\left\vert
u_{n}(t)\right\vert \leq\frac{6\sqrt{\left\vert \alpha_{n}\beta_{n}\right\vert
}}{\left\vert \beta_{n}\right\vert }\left\vert u_{n}(t)\right\vert =O\left(
\frac{1}{n}\right)  .
\]
Now consider the case $4\pi nt>2\varepsilon_{n}.$ Then by Theorem 8 we have
$t>t_{n}$ and in formula (111) we should take $s(t)=1.$ Hence using (111) and
(46) one can conclude that $\left\vert E_{+}(\lambda_{n}(t))\right\vert
>\varepsilon_{n}.$ Therefore from (110) and (131) it follows that
\[
\left\vert v_{n}(t)\right\vert =\frac{\left\vert \alpha_{n}\right\vert
}{\left\vert E_{+}(\lambda_{n}(t))\right\vert }\left\vert u_{n}%
(t)[1]\right\vert \leq\frac{\sqrt{\left\vert \alpha_{n}\right\vert }}%
{\sqrt{\left\vert \beta_{n}\right\vert }}\left\vert u_{n}(t)[1]\right\vert
=O\left(  \frac{1}{n}\right)  .
\]
These estimations for $\left\vert v_{n}(t)\right\vert $ together with (36)
imply (130).
\end{proof}

Now using the last lemma and the formulas (129) \ and (126) we estimate
$d_{n}(t).$

\begin{lemma}
Let $I_{4}$ \ and $I_{5}$ be respectively the set of all $t\in I_{3}$ such
that $4\pi nt$ belongs to the intervals $\left[  \frac{5}{4}\varepsilon
_{n},\left\vert \beta_{n}\right\vert \right]  $ and $(\left\vert \beta
_{n}\right\vert ,4\pi n^{-2}]$, where $I_{k}$ is defined in Remark 5. Let
$\left\vert a\right\vert <\left\vert b\right\vert $.

$\left(  a\right)  $ If $t\in I_{5},$ then $\left\vert d_{n}(t)\right\vert
\sim1.$

$\left(  b\right)  $ If $t\in\left(  I_{1}\cup I_{4}\right)  ,$ then there
exists $c_{7}>0$ such that
\begin{equation}
\left\vert d_{n}(t)\right\vert \geq c_{7}\left\vert a^{n}b^{-n}\right\vert .
\end{equation}

$\left(  c\right)  $ If $t\in I_{2}$ and (102) holds, then (132) is satisfied.
\end{lemma}

\begin{proof}
$\left(  a\right)  $ If $t\in I_{5},$ then it follows from from (111) and
(131) that $\left\vert E_{+}(\lambda_{n}(t))\right\vert \geq\frac{3}%
{2}\left\vert \beta_{n}\right\vert .$ Using it and (114) we obtain
\[
\left\vert v_{n}^{\ast}(t)\right\vert =\left\vert \beta_{n}\right\vert
\left\vert E_{+}(\lambda_{n}(t))\right\vert ^{-1}\left\vert u_{n}^{\ast
}(t)[1]\right\vert \leq\left\vert u_{n}^{\ast}(t)\right\vert .
\]
It with (113) implies that $\left\vert u_{n}^{\ast}(t)\right\vert >2/3$.
Therefore we get the proof of $(a)$ by using (126), (129), and (130) and
taking into account that (119) holds for $\left\vert a\right\vert
\neq\left\vert b\right\vert $ too.

$\left(  b\right)  $ Using (111) and the definitions of $I_{1}$ and $I_{4}$
one can easily see that
\begin{equation}
E_{+}(\lambda_{n}(t))\geq c_{8}\varepsilon_{n}%
\end{equation}
for $t\in\left(  I_{1}\cup I_{4}\right)  $. It with (114) and (131) implies
that%
\[
\left\vert u_{n}^{\ast}(t)\right\vert \geq c_{8}\left\vert v_{n}^{\ast}%
\frac{\sqrt{\left\vert \alpha_{n}\beta_{n}\right\vert }}{\beta_{n}%
[1]}\right\vert \geq c_{8}\left\vert v_{n}^{\ast}\frac{\sqrt{\alpha_{n}}%
}{\sqrt{\beta_{n}}[1]}\right\vert \geq c_{8}\left\vert v_{n}^{\ast}\frac
{a^{n}}{b^{n}}[1]\right\vert .
\]
Therefore using (113) we obtain $\left\vert u_{n}^{\ast}(t)\right\vert \geq
c_{9}\left\vert a^{n}b^{-n}\right\vert .$ Now (132) follows from (126), (129),
(130) and (119).

$\left(  c\right)  $ If (102) holds, then (133) holds for $t\in I_{2}$ too.
Using it and repeating the proof of $(b)$ we get the proof of $(c)$.
\end{proof}

The obtained estimations are enough to prove the main results if (102) holds.
Now we estimate $d_{n}(t)$ for $t\in I_{2}$ when (102) doesn't hold. This case
is the most complicated case. In this case to estimate $\left(  d_{k}%
(t)\right)  ^{-1}$ we use the following formula
\begin{equation}
\frac{-1}{d_{k}(t)}=\frac{\left\Vert \Phi_{t}(\cdot,\lambda_{n}(t))\right\Vert
\left\Vert \Phi_{-t}(\cdot,\lambda_{n}(t))\right\Vert }{\varphi(1,\lambda
_{n}(t))F^{^{\prime}}(\lambda_{n}(t))}=\frac{\left\Vert G_{t}(\cdot
,\lambda_{n}(t))\right\Vert \left\Vert G_{-t}(\cdot,\lambda_{n}(t))\right\Vert
}{\theta^{\prime}(1,\lambda_{n}(t))F^{^{\prime}}(\lambda_{n}(t))}%
\end{equation}
(see (29) and (32) of [17]), where
\[
\Phi_{t}(x,\lambda)=\varphi\theta(x,\lambda)+(e^{it}-\theta)\varphi
(x,\lambda),\text{ }G_{t}(x,\lambda)=\theta^{\prime}\varphi(x,\lambda
)+(e^{it}-\varphi^{\prime})\theta(x,\lambda),
\]
$\varphi=\varphi(1,\lambda),\varphi^{\prime}=\varphi^{\prime}(1,\lambda),$
$\theta=\theta(1,\lambda),$ $\theta^{\prime}=\theta^{\prime}(1,\lambda),$
$\varphi(x,\lambda)$ and $\theta(x,\lambda)$ are defined in (3). Using (3) and
taking into account the Wronskian equality $\theta\varphi^{\prime}%
-\theta^{\prime}\varphi=1$ we obtain $(e^{it}-\theta)(e^{-it}-\theta
)=-\varphi\theta^{\prime}.$ Therefore at least one of the following inequality
holds%
\[
\left\vert \frac{e^{it}-\theta}{\varphi}\right\vert \leq1,\text{ }\left\vert
\frac{e^{-it}-\theta}{\varphi}\right\vert \leq1,\text{ }\left\vert
\frac{e^{it}-\theta}{\theta^{\prime}}\right\vert \leq1,\text{ }\left\vert
\frac{e^{-it}-\theta}{\theta^{\prime}}\right\vert \leq1.
\]
Without loss of generality we suppose that the first equality holds. Then we
have
\[
\frac{\left\Vert \Phi_{t}(\cdot,\lambda_{n}(t))\right\Vert }{\left\vert
\varphi(1,\lambda_{n}(t))\right\vert }\leq\left\Vert \theta(\cdot,\lambda
_{n}(t))\right\Vert +\left\Vert \varphi(\cdot,\lambda_{n}(t))\right\Vert .
\]
It with (65) and the following asymptotic formulas (see page 63 of [3])
\[
\theta(x,\lambda)=\cos\mu x+\frac{\sin\mu x}{2\mu}Q(x)+\frac{\cos\mu x}%
{4\mu^{2}}(q(x)-q(0)-\frac{1}{2}Q^{2}(x))+O\left(  \mu^{-3}\right)  ,
\]%
\[
\varphi(x,\lambda)=\frac{\sin\mu x}{\mu}-\frac{\cos\mu x}{2\mu^{2}}%
Q(x)+\frac{\sin\mu x}{4\mu^{3}}(q(x)+q(0)-\frac{1}{2}Q^{2}(x))+O(\mu^{-4}),
\]
where $\mu=\sqrt{\lambda},$ $\left\vert \operatorname{Im}\mu\right\vert <3,$
$Q(x)=\int_{0}^{x}q(x)dx$ implies the following inequalities%
\[
\frac{\left\Vert \Phi_{t}(\cdot,\lambda_{n}(t))\right\Vert }{\left\vert
\varphi(1,\lambda_{n}(t))\right\vert }<c_{10},\text{ }\left\Vert \Phi
_{-t}(\cdot,\lambda_{n}(t))\right\Vert <c_{10}n^{-4}%
\]
for $t\in I_{2}.$ Therefore using the substitution $\lambda=\mu^{2},$
$f(\mu)=F(\lambda),$ $F^{^{\prime}}(\lambda)=f^{^{\prime}}(\mu)\frac{1}{2\mu
},$ we get
\begin{equation}
\frac{1}{\left\vert d_{n}(t)\right\vert }\leq\frac{c_{11}n^{-3}}{\left\vert
f^{^{\prime}}(\mu_{n}(t))\right\vert },
\end{equation}
where $f^{^{\prime}}=\frac{df}{d\mu},$ $\mu_{n}(t)=\sqrt{\lambda_{n}(t)}%
\in\left\{  z\in\mathbb{C}\text{: }\left\vert z-2\pi n\right\vert <2\right\}
=:\Omega(2\pi n,2).$

Now using the well-known asymptotic formula for the Hill discriminant
$f(\mu)=F(\lambda):$
\begin{equation}
f(\mu)=2\cos\mu+R(\mu),\text{ }\left\vert R(\mu)\right\vert <c_{12}\left\vert
\mu\right\vert ^{-3}%
\end{equation}
(see page 64 of [3]) and the Cauchy's integral formula
\begin{equation}
f^{(k)}(\mu_{n}(t))=\frac{k!}{2\pi i}\int_{\gamma}\frac{f(\xi)}{\left(
\xi-\mu_{n}(t)\right)  ^{k+1}}d\xi,
\end{equation}
where $R(\mu)$ is an entire function and $\mu\in\Omega(2\pi n,2)$,
$\gamma=\left\{  z\in\mathbb{C}\text{: }\left\vert z-\mu_{n}(t)\right\vert
=1\right\}  $ and $\mu_{n}(t)\in\Omega(2\pi n,1)$ we estimate $\left\vert
f^{^{\prime}}(\mu_{n}(t))\right\vert $. \ Using (136) and (137) for $k=1$ we
obtain
\begin{equation}
f^{^{\prime}}(\mu)=-2\sin\mu+R^{^{\prime}}(\mu),\text{ }\left\vert
R^{^{\prime}}(\mu)\right\vert <c_{13}\left\vert \mu\right\vert ^{-3}.
\end{equation}
By (138) and Rouche's theorem the equation $f^{^{\prime}}(\mu)=0$ has a unique
root $\mu_{n}$ in $c_{14}n^{-3}$ neighborhood on $2\pi n$ for large $n.$
Moreover, using (136) and (3) we see that $\mu_{n}=\mu_{n}(t_{0})$,
$\left\vert t_{0}\right\vert <c_{15}n^{-3/2}$ and $\left(  \mu_{n}%
(t_{0})\right)  ^{2}$ is a multiple eigenvalue of the operator $H_{t_{0}%
}(a,b)$ satisfying $\left\vert \mu_{n}(t_{0})-2\pi n\right\vert <c_{14}%
n^{-3}.$ It with (65) gives
\begin{equation}
\text{ }\left\vert \mu_{n}(t)-\mu_{n}(t_{0})\right\vert <c_{16}n^{-3}%
\end{equation}
for $t\in I_{2}$. Since the proof of (81) is unchanged if $[0,\rho]$ is
replaced by $\left\{  \left\vert t\right\vert <c_{15}n^{-3/2}\right\}  $ we
have
\begin{equation}
(4\pi nt_{0})^{2}=-\beta_{n}\alpha_{n}\left(  1+O(n^{-2})\right)  .
\end{equation}

\begin{remark}
As is noted in above we consider the case when (102) does not hold say for
$c_{3}=1/9.$ Then $-\beta_{n}\alpha_{n}=u_{n}+iv_{n},$ $\left\vert
u_{n}\right\vert \geq9\left\vert v_{n}\right\vert .$ Since $\mu_{n}%
(-t_{0})=\mu_{n}(t_{0}),$ without using of generality it can be assumed that
$\operatorname{Re}t_{0}\geq0.$ These arguments with (140) imply that
$t_{0}=u_{0}+iv_{0},$ $u_{0}>8\left\vert v_{0}\right\vert .$ Using it, (140)
and definition of $I_{2}$ we obtain the following relations which we use in
the proof of the main results
\begin{equation}
\text{ }\left\vert t_{0}\right\vert =\left(  4\pi n\right)  ^{-1}%
\varepsilon_{n}(1+O(n^{-2})),\text{ }\tfrac{3}{4}\left\vert t_{0}\right\vert
\leq u_{0}\leq\left\vert t_{0}\right\vert ,\text{ }\left\vert t-t_{0}%
\right\vert <\left\vert t_{0}\right\vert
\end{equation}
for all $t\in I_{2},$ where $\varepsilon_{n}$ is defined in Remark 5.
\end{remark}

\begin{lemma}
If $t\in I_{2}$ , then%
\begin{equation}
\frac{1}{\left\vert d_{k}(t)\right\vert }\leq\frac{c_{11}n^{-3}}{\left\vert
\sqrt{\left\vert \sin t_{0}\right\vert }\sqrt{\left\vert t-t_{0}\right\vert
}\right\vert }.
\end{equation}

\end{lemma}

\begin{proof}
By formulas (136) and (137) for $k=2$ we have
\begin{equation}
f^{^{\prime\prime}}(\mu)=-2\cos\mu+R^{^{\prime\prime}}(\mu),\text{ }\left\vert
R^{^{\prime\prime}}(\mu)\right\vert <c_{17}\left\vert \mu\right\vert ^{-3}%
\end{equation}
for $\left\vert \mu-\mu_{n}(t_{0})\right\vert <c_{18}n^{-3}.$ Using the
Taylor's theorem for $f^{^{\prime}}(\mu)$ and taking into account that
$f^{^{\prime}}(\mu_{n}(t_{0}))=0$ we get
\[
f^{^{\prime}}(\mu_{n}(t))=f^{^{\prime\prime}}(\mu_{n}(t))(\mu_{n}(t)-\mu
_{n}(t_{0}))+f_{2}(\mu_{n}(t))(\mu_{n}(t)-\mu_{n}(t_{0}))^{2},
\]
where $\left\vert f_{2}(\mu_{n}(t))\right\vert <c_{19}\left\vert \mu
_{n}(t)\right\vert ^{-1}$ (see pages 125 and 126 of [1]). It with (143), (139)
and (65) imply that%
\begin{equation}
\left\vert f^{^{\prime}}(\mu_{n}(t))\right\vert >\left\vert \mu_{n}(t)-\mu
_{n}(t_{0})\right\vert .
\end{equation}
Similarly, using the Taylor's theorem for $f(\mu)$ and $\cos t$ and taking
into account that $f(\mu_{n}(t_{0}))=2\cos t_{0},$ $f^{^{\prime}}(\mu
_{n}(t_{0}))=0,$ $f^{^{\prime\prime}}(\mu_{n}(t))=-2+O(n^{-1})$ we obtain%
\[
f(\mu_{n}(t)))=2\cos t_{0}-(\mu_{n}(t))-\mu_{n}(t_{0}))^{2}(1+o(1)),
\]%
\[
2\cos t=2\cos t_{0}-2\left(  \sin t_{0}\right)  (t-t_{0})-(t-t_{0})^{2}\left(
1+o(1)\right)  .
\]
These equalities with $f(\mu_{n}(t))=2\cos t$ and (141)) imply that
\[
\left\vert \mu_{n}(t)-\mu_{n}(t_{0})\right\vert >\sqrt{\left\vert \sin
t_{0}\right\vert }\sqrt{\left\vert t-t_{0}\right\vert }.
\]
It with (135) and (144) implies (142).
\end{proof}

\begin{theorem}
If $ab\neq0,$ then the operator $H(a,b)$ has no ESS at infinity.
\end{theorem}

\begin{proof}
Using Lemma 4 and the definitions of $I_{1},$ $I_{4},$ $I_{5},$ $\varepsilon
_{n}$ and $\beta_{n}$ one can easily verify that
\begin{equation}%
{\textstyle\int\limits_{\left[  0,n^{-3}\right]  \backslash I_{2}}}
\left\vert d_{n}(t)\right\vert ^{-1}dt=O(n^{-3}).
\end{equation}
On the other hand, by Proposition 2 the integral of $\left\vert d_{n}%
(t)\right\vert ^{-1}$ over $\left[  n^{-3},\rho\right]  $ is less than
$c_{20}.$ If (102) holds then by Lemma 3(c) in (145) the integral of
$\left\vert d_{n}(t)\right\vert ^{-1}$ over $\left[  0,n^{-3}\right]
\backslash I_{2}$ can be replaced by the integral over $\left[  0,n^{-3}%
\right]  $. If (102) does not hold, then using (142), (141) and the obvious
relations $I_{2}\subset\left[  0,2\left\vert t_{0}\right\vert \right]  $ and
$\left\vert t-u_{0}\right\vert \leq\left\vert t-t_{0}\right\vert $ we obtain%
\[
\int\limits_{I_{2}}\left\vert d_{n}(t)\right\vert ^{-1}dt\leq\int
\limits_{\left[  0,2\left\vert t_{0}\right\vert \right]  }\frac{c_{17}n^{-3}%
}{\left\vert \sqrt{\left\vert \sin t_{0}\right\vert }\sqrt{\left\vert
t-u_{0}\right\vert }\right\vert }dt=O\left(  n^{-3}\right)  .
\]
Thus the integral of $\left\vert d_{n}(t)\right\vert ^{-1}$ over $\left[
0,\rho\right]  $ is less than $c_{21}.$ Similarly, integral of $\left\vert
d_{n}(t)\right\vert ^{-1}$ over $\left[  \pi-\rho,\pi\right]  $ is less than
$c_{21}.$ These inequalities with (17) imply that the integral of $\left\vert
d_{n}(t)\right\vert ^{-1}$ over $[0,\pi]$ is less than $c_{22}.$ Since
$\lambda_{n}(-t))=\lambda_{n}(t))$ (see Remark 1) it follows from (134) that
$\left\vert d_{n}(-t)\right\vert ^{-1}=\left\vert d_{n}(t)\right\vert ^{-1}.$
Therefore the integral of $\left\vert d_{n}(t)\right\vert ^{-1}$ over
$(-\pi,\pi]$ is less than $2c_{22}$ and hence by Definition 5 the operator
$H(a,b)$ has no ESS at infinity.
\end{proof}

\textbf{The proofs of Theorems 2, 3 and 4. } The proofs of Theorems 2 and 3
follow from Theorem 12 and Summary 1. Now we prove Theorem 4. It is well-known
that (see [4], [7] and [15]) if either $a=0$ or $b=0$, then $\lambda
_{n}(0)=\left(  2\pi n\right)  ^{2}$ for $n\in\mathbb{Z}\backslash\left\{
0\right\}  $ and $\lambda_{n}(\pi)=\left(  2\pi n+\pi\right)  ^{2}$ for
$n\in\mathbb{Z}$ are the double 2-periodic eigenvalues. Moreover in [7] it was
proven that the geometric multiplicities of these eigenvalues is 1. Thus by
Summary 1$(b),$ $\lambda_{n}(0)$ for $n\in\mathbb{Z}\backslash\left\{
0\right\}  $ and $\lambda_{n}(\pi)$ for $n\in\mathbb{Z}$ are ESS. Moreover it
readily follows from Definitions 4 and 5 that if $L(q)$ has infinitely many
ESS converging to infinity, then it has ESS an infinity. In [18], we have
proved that (14) holds, where%

\begin{equation}
2\pi f_{0}=\int\limits_{[-h,h]}\sum_{\left\vert n\right\vert \leq N(h)}%
a_{n}(t)\Psi_{n,t}dt+\sum_{n>N(h)}\int\limits_{[-h,h]}\left(  a_{n}%
(t)\Psi_{n,t}+a_{-n}(t)\Psi_{-n,t}\right)  dt,
\end{equation}%
\[
2\pi f_{\pi}=\int\limits_{[\pi-h,\pi+h]}\sum_{n=-N(h)-1}^{N(h)}a_{n}%
(t)\Psi_{n,t}dt+\sum_{n>N(h)}\int\limits_{[\pi-h,\pi+h]}\left(  a_{n}%
(t)\Psi_{n,t}+a_{-(n+1)}(t)\Psi_{-(n+1),t}\right)  dt.
\]
By Summary 1(e) if $k\neq0,$ then the sum of two expressions $a_{k}%
(t)\Psi_{k,t}(x)$ and $a_{-k}(t)\Psi_{-k,t}(x)$ corresponding to the ESS
$\lambda_{k}(0)$ is integrable on $[-h,h],$ while both of them is
nonintegrable. Besides $a_{0}(t)\Psi_{0,t}$ is integrable since $\lambda
_{0}(0)$ is a simple eigenvalue and hence is not an ESS. Therefore we have
\[
\int\limits_{\lbrack-h,h]}\sum_{\left\vert n\right\vert \leq N(h)}a_{n}%
(t)\Psi_{n,t}dt=\int\limits_{[-h,h]}a_{0}(t)\Psi_{0,t}dt+\sum_{n=1}^{N(h)}%
\int\limits_{[-h,h]}\left(  a_{n}(t)\Psi_{n,t}+a_{n}(t)\Psi_{n,t}\right)  dt.
\]
Using it in (146) we get (15). In the same way from the last equality for
$f_{\pi}$ we obtain (16).

\end{document}